\documentclass[fleqn,11pt,letterpaper]{article}
\usepackage{amsfonts}
\usepackage{amssymb}
\usepackage{amsmath}
\usepackage{enumitem}
\usepackage{multirow}
\usepackage{float}
\usepackage{graphicx}
\usepackage{color}
\usepackage[margin=2.57cm]{geometry}
\definecolor{darkred}{rgb}{0.5,0.2,0.2}
\usepackage[pdftitle={Local fractional bootstrap},pdfauthor={Mikkel Bennedsen, Ulrich Hounyo, Asger Lunde and Mikko Pakkanen},colorlinks=TRUE,allcolors=darkred]{hyperref}
\usepackage{booktabs}
\usepackage{pdflscape}
\usepackage[round]{natbib}
\usepackage{tabularx}

\usepackage{float}
\floatstyle{plaintop}
\restylefloat{table}

\newtheorem{theorem}{Theorem}[section]

\newtheorem{example}[theorem]{Example}

\newtheorem{lemma}{Lemma}[section]

\newtheorem{proposition}{Proposition}[section]
\newtheorem{remark}{Remark}

\newenvironment{proof}[1][Proof]{\textbf{#1.} }{\ \rule{0.5em}{0.5em}}
\def\BSS{{\mathcal BSS}}
\def\R{{\mathbb R}}
\def\N{{\mathbb N}}

\newcommand\mikko[1]{\textcolor{black}{#1}}

\newcommand\mikkel[1]{\textcolor{black}{#1}}

\begin{document}

\title{The Local Fractional Bootstrap\footnote{We would like to thank Solveig S\o rensen for competent and precise proof reading of the manuscript. Our research has been supported by CREATES (DNRF78), funded by the Danish National Research Foundation, 
by Aarhus University Research Foundation (project ``Stochastic and Econometric Analysis of Commodity Markets"), and
by the Academy of Finland (project 258042).}}
\author{Mikkel Bennedsen\thanks{
Department of Economics and Business Economics and CREATES, Aarhus
University, Fuglesangs All\'{e} 4, 8210 Aarhus V, Denmark. E-mail:\ \href{mailto:mbennedsen@econ.au.dk}%
{\nolinkurl{mbennedsen@econ.au.dk}}} \and
Ulrich Hounyo\thanks{
Department of Economics and Business Economics and CREATES, Aarhus
University, Fuglesangs All\'{e} 4, 8210 Aarhus V, Denmark. E-mail:\ \href{mailto:uhounyo@econ.au.dk}%
{\nolinkurl{uhounyo@econ.au.dk}} } \and Asger Lunde\thanks{
Department of Economics and Business Economics and CREATES, Aarhus
University, Fuglesangs All\'{e} 4, 8210 Aarhus V, Denmark. E-mail:\ \href{mailto:alunde@econ.au.dk}%
{\nolinkurl{alunde@econ.au.dk}} } \and Mikko S. Pakkanen\thanks{
Department of Mathematics, Imperial College London, South Kensington Campus,
London SW7 2AZ, UK and CREATES, Aarhus University\mikko{, Denmark}. E-mail:\ \href{mailto:m.pakkanen@imperial.ac.uk}%
{\nolinkurl{m.pakkanen@imperial.ac.uk}} }}
\maketitle

\setlength{\baselineskip}{18pt}\setlength{\abovedisplayskip}{10pt} %
\belowdisplayskip\abovedisplayskip\setlength{\abovedisplayshortskip }{5pt} %
\abovedisplayshortskip\belowdisplayshortskip\setlength{%
\abovedisplayskip}{8pt} \belowdisplayskip\abovedisplayskip%
\setlength{\abovedisplayshortskip }{4pt}

\begin{abstract}
We \mikko{introduce a bootstrap procedure for high-frequency statistics of Brownian semistationary
processes. More specifically, we focus on a hypothesis test on the
roughness of sample paths of Brownian semistationary processes, which uses an estimator based on a ratio of realized power variations. Our new
resampling method, the local fractional bootstrap, relies on
simulating an auxiliary fractional Brownian motion that mimics the
fine properties of high frequency differences of the Brownian semistationary process under
the null hypothesis. We prove the first order validity of the bootstrap method and in simulations we
observe that the bootstrap-based hypothesis test provides considerable finite-sample improvements over an existing test that is based on a central limit theorem.} This is important when studying the roughness properties of time series data; we illustrate this by applying the bootstrap method to two empirical \mikko{data sets}: \mikko{we assess the roughness of a time series of high-frequency asset prices and we test the validity of Kolmogorov's scaling law in atmospheric turbulence data.}
\end{abstract}


\vspace*{1em}


\vspace*{1em}


\section{Introduction}
In the study of pathwise properties of continuous \mikko{stochastic processes}, \emph{roughness} is a central \mikko{attribute}. \mikko{Theoretically}, roughness relates to the degree of H\"older \mikko{regularity} enjoyed by the \mikko{sample paths of the stochastic process in} question. \mikko{The fractional Brownian motion (fBm), introduced by \citet{Kolmogorov1940} and popularized later by \citet{MVN68}, is perhaps the most well-known example of a process that can exhibit various degrees of roughness. The fBm with Hurst index $H \in (0,1)$ is a Gaussian process that coincides with the standard Brownian motion when $H = 1/2$. If $H < 1/2$ (respectively $H >1/2$), then the sample paths of the fBm are rougher (respectively smoother) than those of the standard Brownian motion in terms of H\"older regularity.}
In this work, we are concerned with conducting inference on the \emph{fractal index}, $\alpha$, of a \mikko{stochastic process}, when $\alpha$ is estimated using the so-called \emph{change-of-frequency} (COF) estimator, \mikko{introduced by \cite{LR2001} for Gaussian processes and extended by \cite{BNCP12} and \cite{CHPP13} to a non-Gaussian setting. In the case of the fBm} it holds that $\alpha = H - 1/2$, \mikko{whilst in general $\alpha < 0$ indicates roughness and $\alpha > 0$ smoothness relative to the standard Brownian motion, as with fBm. When $\alpha = 0$, the stochastic process} under consideration has the same roughness as the \mikko{standard} Brownian motion. 

Several interesting \mikko{empirical time} series exhibit signs of roughness, i.e. $%
\alpha <0$. \mikko{Some noteworthy examples include:
\begin{itemize}
\item time-wise measurements of velocity in turbulent flows \citep{CHPP13}, where roughness in inertial time scales is predicted by Kolmogorov's scaling law \citep{kolmogorov41} and Taylor's \emph{frozen field hypothesis} \citep{Taylor1938},
\item time series of electricity spot prices \citep{BNBV13,bennedsen15},
\item measures of the realized volatility of asset prices \citep{GJR14,BLP16}.
\end{itemize}}
In these
applications, estimation of, and inference on, the index $\alpha$
is important. 
\mikko{There is a long history of methods of estimating $\alpha$, concentrating mostly on Gaussian processes, of which a comprehensive survey is provided in \citet{GSP12}.
In the time series data mentioned above, non-Gaussian features are pervasive, however, which is why we concentrate on a specific, yet flexible, non-Gaussian framework. \emph{Brownian semistationary} ($\BSS$) processes \citep{BNSc07,BNSc09} form a class of stochastic processes that accommodate various departures from Gaussianity and different degrees of roughness. \cite{BNCP12} and \cite{CHPP13} and have studied the properties of the COF estimator of $\alpha$ in the context of $\BSS$ processes. In particular, they have derived a central limit theorem (CLT), that makes it possible to conduct hypothesis tests on $\alpha$.}


In the present \mikko{paper,} our main focus will be on the COF estimator of \mikko{a driftless $\BSS$ process $(X_t)_{t \in \R}$}, given by
\begin{align*}
X_t = \int_{-\infty}^t g(t-s) \sigma_s d W_s,
\end{align*}
where \mikko{$g : \R^+ \rightarrow \R$} is a kernel function, \mikko{$(\sigma_t)_{t \in \R}$} a stochastic volatility process, and \mikko{$(W_t)_{t \in \R}$ a standard} Brownian motion. Important for our present purpose, we assume that the kernel function \mikko{$g(x)$} behaves \mikko{like} a power function, $x \mapsto x^{\alpha}$, when \mikko{$x$ is near zero}; a statement we will make precise below. When this assumption holds, and some additional (mild) technical conditions are met, the sample paths of $X$ are H\"older continuous with index $\alpha+1/2-\varepsilon$ for any $\varepsilon>0$ \citep[Proposition 2.1]{BLP15}. Moreover, $\alpha=0$ \mikkel{is a necessary condition for the process to be a semimartingale \citep{basse08,bennedsen16a}}. Intuitively, the $\BSS$ process is a moving average process driven by volatility-modulated Brownian noise and is, thus, quite general and flexible. 
\mikko{The $\BSS$ framework is also closely related to processes such as the fBm and Gaussian Volterra processes of convolution type; see e.g. \cite{BLP15}.} Therefore we expect the methods proposed in this paper to apply to these processes as well, but research into the specifics is \mikko{beyond} the scope of the present paper and left for future work.

The contribution of this paper is to derive a bootstrap procedure that \mikko{improves the
finite sample properties of the test of the null hypothesis
\begin{equation*}
H_{0}:\alpha
=\alpha _{0},
\end{equation*}
for some $\alpha _{0}\in \left( -\frac{1}{2},\frac{1}{2}%
\right)$, when the fractal index  $\alpha$ is estimated using the COF estimator.} 
Theoretically, the COF estimator \mikko{has} two regimes: \mikko{in the first regime} $%
\alpha \in \left( -\frac{1}{2},\frac{1}{4}\right)$, the estimator uses the
entire sample to estimate $\alpha$. In this case, we propose a novel
bootstrap method, the \emph{local fractional bootstrap}, which utilizes
simulations of an auxiliary fractional Brownian motion with \mikko{Hurst index $H=\alpha _{0}+\frac{1}{2}
$,} 
thereby mimicking the \mikko{fine properties} of the \mikko{sample paths} of the underlying $\BSS$ process
under \mikko{$H_0$}. We establish \mikko{the} first-order asymptotic validity of the local
fractional bootstrap for the percentile-$t$ methods, i.e. when the
test-statistic is normalized by its (bootstrap) standard deviation.

As noted in \mikko{\citet[Section 4]{CHPP13}, in the second regime} $\alpha \in \left[ \frac{1}{4},\frac{1}{2}\right) $, \mikko{the asymptotic behavior of the COF estimator is potentially affected (depending on the form of $g$) by a non-negligible bias term, causing the CLT to break down.} For this reason, the authors suggest using a modified COF estimator \mikko{that implements asymptotically increasing gaps between increments from which the power variations are computed.} These gaps make the increments used in the estimator asymptotically uncorrelated, which \mikko{opens the door to} a \emph{wild bootstrap}  approach \citep[e.g.][]{wu86,liu88,{GM09}} in this regime. However, since the local fractional bootstrap method proposed in this paper works even when $\alpha \in \left[ \frac{1}{4},\frac{1}{2}\right)$, we, for the sake of brevity, relegate the details on the wild bootstrap method, along with the simulation study of its finite sample properties, to a separate web appendix \citep{BHLP16}.


In a Monte Carlo simulation study, we \mikko{assess} the finite sample properties of the local fractional bootstrap  \mikko{procedure in comparison with the inference method based on the} CLT of \cite{CHPP13}. We find that \mikko{for all $\alpha \in \left(-\frac{1}{2},\frac{1}{2} \right)$}, the local fractional bootstrap offers improvements \mikko{in terms of} the size of the test of $H_0$, especially when the sample size \mikko{ranges from} small to
moderate. Indeed, since our method simulates the auxiliary bootstrap observations under \mikko{$H_0$}, we minimize the probability of a type I error \citep{DM99}, i.e. of rejecting $H_0$ when \mikko{it actually holds}.  This \mikko{feature} proves to be important when we in the empirical section apply the \mikko{method} to assess the roughness of \mikko{the} time series of \mikko{logarithmic} prices of the E-mini futures contract. In this case, no-arbitrage considerations \mikko{suggest that} $\alpha = 0$, and we find that the bootstrap procedure is crucial for achieving the correct size of the test of $H_0: \alpha = 0$ when applied to intraday price series. We also apply the bootstrap method to a \mikko{time series of measurements of atmospheric turbulence} to test for the \mikko{empirical validity of} \mikko{Kolmogorov's scaling law} \citep{kolmogorov41}. We find \mikko{the data to be in good agreement} with the scaling law, \mikko{but again using the bootstrap is crucial for accurate inference when the sample size is small.}

The rest of this paper is structured as follows. Section \ref%
{sec:setupreview} sets the stage by presenting the mathematical definition
of the $\BSS$ process as well as the assumptions we work under. This
section also briefly reviews existing results as they pertain to the present
work. In Section \ref{sec:fracBoot} we detail our bootstrap method, the
local fractional bootstrap, and give the details \mikko{on} its implementation. Section \ref%
{sec:MC} contains a Monte Carlo study of the finite sample properties of the bootstrap method and Section %
\ref{sec:empirical} presents the empirical applications. Section \ref%
{sec:concl} concludes. Simulation setup, proofs, as well as some additional technical derivations, are
given in \mikko{Appendices \ref{sec:sim}, \ref{app:derive}, and \ref{sec:proofs}.} \mikko{The details} on the wild bootstrap method, including proofs and a simulation \mikko{experiment}, are available in a web appendix \citep{BHLP16}.

\section{Setup, assumptions, and review of existing results}
\label{sec:setupreview}

\mikko{Now, we introduce some essential notation. Following the conventions of} bootstrap literature,
$\mathbb{P}^{\ast }$ ($\mathbb{E}^{\ast }$ and $Var^{\ast }$) denotes the
probability measure (expected value and variance) induced by the bootstrap
resampling, conditional on a realization of the original time series. In
addition, for a sequence of bootstrap statistics $Z_{n}^{\ast }$, we write $%
Z_{n}^{\ast }=o_{p^{\ast }}\left( 1\right) $ in probability, or $Z_{n}^{\ast
}\overset{\mathbb{P}^{\ast }}{\rightarrow }0$, as $n\rightarrow \infty $, in
probability, if for any $\varepsilon >0$, $\delta >0$, 
\begin{align*}
\lim_{n\rightarrow
\infty }\mathbb{P}\left[ \mathbb{P}^{\ast }\left( \left\vert Z_{n}^{\ast
}\right\vert >\delta \right) >\varepsilon \right] =0. 
\end{align*}
Similarly, we write $%
Z_{n}^{\ast }=O_{p^{\ast }}\left( 1\right) $ as $n\rightarrow \infty $, in
probability if for all $\varepsilon >0$ there exists $M_{\varepsilon
}<\infty $ such that 
\begin{align*}
\lim_{n\rightarrow \infty }\mathbb{P}\left[ \mathbb{P}%
^{\ast }\left( \left\vert Z_{n}^{\ast }\right\vert >M_{\varepsilon }\right)
>\varepsilon \right] =0. 
\end{align*}
Finally, we write $Z_{n}^{\ast }\overset{d^{\ast }}%
{\rightarrow }Z$ as $n\rightarrow \infty $, in probability, if conditional
on the sample, $Z_{n}^{\ast }$ converges weakly to $Z$ under $\mathbb{P}%
^{\ast },$ for all samples contained in a set with $\mathbb{P}$-probability 
converging to one.

\subsection{$\BSS$ setup and assumptions}

\label{sec:setup}

We follow \cite{BNBV13} and consider a filtered
probability space $\left( \Omega ,\mathcal{F},\left( \mathcal{F}_{t}\right)
_{t\in \mathbb{%
\mathbb{R}
}},\mathbb{P}\right) $, on which we define a Brownian semistationary ($\BSS$)%
\textit{\ }process\textit{\ }$X=\left( X_{t}\right) _{t\in \mathbb{%
\mathbb{R}
}}$, without a drift as%
\begin{equation}
X_{t}=\int_{-\infty }^{t}g\left( t-s\right) \sigma _{s}dW_s, \quad t \in \R,
\label{one}
\end{equation}%
where \mikko{$(W_t)_{t \in \R}$} is a \mikko{two-sided standard} Brownian motion, $g:\mathbb{R}^{+}\rightarrow \mathbb{R}$ is a deterministic weight
function satisfying $g\in L^{2}\left( \mathbb{%
\mathbb{R}
}^{+}\right) $, and \mikko{$(\sigma_t)_{t \in \R}$} is an $\left( \mathcal{F}_{t}\right) _{t\in
\mathbb{%
\mathbb{R}
}}$-adapted c\`{a}dl\`{a}g process. We assume that
\begin{equation*}
\int_{-\infty }^{t}g^{2}\left( t-s\right) \sigma _{s}^{2}ds<\infty \text{
a.s.}, \quad \mikko{\textrm{for all $t \in \R$}},
\end{equation*}%
to ensure that \mikko{$X_{t}$ is a.s.\ finite} for any $t\in \mathbb{%
\mathbb{R}
}$. We introduce a centered stationary Gaussian process $G=\left(
G_{t}\right) _{t\in \mathbb{%
\mathbb{R}
}}$ that is associated to $X$, which we will call the \emph{Gaussian core} of $X$,
as%
\begin{equation}
G_{t}=\int_{-\infty }^{t}g\left( t-s\right) dW_s ,\quad t\in
\mathbb{%
\mathbb{R}
}\text{.}  \label{Gaussian}
\end{equation}%
The correlation kernel $r$ of $G$ is given via%
\begin{equation*}
r\left( t\right) =corr\left( G_{s},G_{s+t}\right) =\frac{\int_{0}^{\infty
}g\left( u\right) g\left( u+t\right) du}{\left\Vert g\right\Vert
_{L^{2}\left( \mathbb{%
\mathbb{R}
}^{+}\right) }^{2}},\quad t\geq 0.
\end{equation*}%
A crucial \mikko{object in} the asymptotic theory is the variogram $R$, given
by%
\begin{equation*}
R\left( t\right) =\mathbb{E}\left[ \left( G_{s+t}-G_{s}\right) ^{2}\right]
=2\left\Vert g\right\Vert _{L^{2}\left( \mathbb{%
\mathbb{R}
}^{+}\right) }^{2}\left( 1-r\left( t\right) \right), \quad t \geq 0.
\end{equation*}%
We assume that the process $X$ is observed at equidistant time points $%
t_{i}=i\Delta _{n},$ $i=0,1,\ldots ,\left\lfloor t/\Delta _{n}\right\rfloor $%
, with $\Delta_n \downarrow 0$ as $n\rightarrow \infty$. This kind of asymptotics is termed \emph{in-fill asymptotics}. The theory considered in this paper will call for computing second order differences of the $\BSS$ process using different lag spacing, $\upsilon \in \N$. In particular, we are concerned with power variations of the following type
\begin{equation}
V\left( X;p,\upsilon\right)^n _{t}\equiv
\sum\limits_{i=2 \upsilon}^{\left\lfloor t/\Delta _{n}\right\rfloor
}\left\vert X_{i\Delta_n} - 2 X_{(i-\upsilon)\Delta_n} +  X_{(i-2\upsilon)\Delta_n} \right\vert ^{p},  \label{Power-Var}
\end{equation}%
where \mikko{$p\geq 1$} and where we refer to $\upsilon$ as the lag between observations. Although the theory goes through for general $\upsilon \in \N$ we will mainly consider $\upsilon = 1, 2$, which will be sufficient for our purposes. For the asymptotic theory,  \cite{CHPP13} also introduce the normalized power variations:
\begin{equation}
\bar{V}\left( X;p,\upsilon\right)^n _{t}\equiv \Delta _{n}\tau
_{n}\left( \upsilon \right) ^{-p}V\left( X;p,\upsilon \right)^n_{t},  \label{Power-Var-Bar}
\end{equation}%
where $\tau _{n}\left( \upsilon\right) =\sqrt{\mathbb{E}\left[
\left\vert G_{i\Delta_n} - 2 G_{(i-\upsilon)\Delta_n} +  G_{(i-2\upsilon)\Delta_n}\right\vert ^{2}\right] }$ is the standard deviation of the second order increment of the Gaussian core calculated with lag spacing $\upsilon\Delta_n$. 

Our proposal is to use the bootstrap \mikko{to approximate the sampling distributions of} a general class of nonlinear transformations of these
statistics. \mikko{This relates to} the limiting behavior of the \mikko{roughness}
parameter estimator of the $\BSS$ process, \mikko{studied in} 
\cite{CHPP13}. In order to \mikko{recall} the consistency result \mikko{for} \ $%
\bar{V}\left( X;p,\upsilon \right)^n_{t}$, derived by \cite{CHPP13}, we need to introduce a set of assumptions. Below, $\alpha $ denotes
a number in $\left( -\frac{1}{2},0\right) \cup \left( 0,\frac{1}{2}\right) $ and functions $L_f$ indexed by a mapping $f$, are assumed to be slowly varying at zero, i.e. be such that $\lim_{x\downarrow 0} \frac{L_f(tx)}{L_f(x)} = 1$ for all $t>0$. For a function $f$, $f^{(k)}$ denotes the \mikko{$k$-th} derivative of $f$.

\begin{description}
\item[Assumption 1.]

\item[(i)] $g\left( x\right) =x^{\alpha }L_{g}\left( x\right) $.

\item[(ii)] $g^{\left( 2\right) }\left( x\right) =x^{\alpha -2}L_{g^{\left(
2\right) }}\left( x\right) $ and for any $\epsilon >0,$ we have $g^{\left(
2\right) }\in L^{2}\left( \left( \epsilon ,\infty \right) \right) .$
Furthermore, $\left\vert g^{\left( 2\right) }\right\vert $ is non-increasing
on the interval $\left( a,\infty \right) $ for some $a>0.$

\item[(iii)] For any $t>0$%
\begin{equation*}
F_{t}=\int_{1}^{\infty }\left\vert g^{\left( 2\right) }\left( s\right)
\right\vert ^{2}\sigma _{t-s}^{2}ds<\infty
\end{equation*}%
almost surely. 
\end{description}

The next set of assumptions deals with the variogram $R$.

\begin{description}
\item[Assumption 2.] For the \mikko{roughness parameter} $\alpha $ from Assumption
1, it holds that

\item[(i)] $R\left( x\right) =x^{2\alpha +1}L_{R}\left( x\right) $.

\item[(ii)] $R^{\left( 4\right) }\left( x\right) =x^{2\alpha
-3}L_{R^{\left( 4\right) }}\left( x\right) $.

\item[(iii)] There exists a $b\in \left( 0,1\right) $ such that%
\begin{equation*}
\underset{x\downarrow 0}{\lim \sup }\underset{y\in \left[ x,x^{b}\right] }{%
\sup }\left\vert \frac{L_{R^{\left( 4\right) }}\left( y\right) }{%
L_{R}\left( x\right) }\right\vert <\infty .
\end{equation*}%

Finally, we introduce an assumption on the smoothness of the process $\sigma
.$

\item[Assumption 3.] For any $q>0$, it holds that%
\begin{equation*}
\mathbb{E}\left[ \left\vert \sigma _{t}-\sigma _{s}\right\vert ^{q}\right]
\leq C_{q}\left\vert t-s\right\vert ^{\gamma q}
\end{equation*}%
for some \mikko{$\gamma >1/2$} and $C_{q}>0$.
\end{description}

\begin{remark}
\mikko{The methods and results presented in this paper can be trivially extended to processes of the form
\begin{equation*}
X^\sharp_t = \int_{-\infty}^t \big(g(t-s) - g_0(-s)\big) \sigma_s dW_s, \quad t \geq 0,
\end{equation*}
where $g$ is as before and $g_0 : \R \rightarrow \R$ is a measurable function such that $g(x)=0$ for all $x<0$ and 
\begin{equation*}
 \int_{-\infty}^t \big(g(t-s) - g_0(-s)\big)^2 \sigma^2_s ds< \infty, \quad \textrm{for all $t \geq 0$}.
\end{equation*}
In particular, we note that $X^\sharp_t-X^\sharp_s = X_t - X_s$ for any $t>s \geq 0$, which is why the techniques that rely on the increments of $X$ presented below, apply, mutatis mutandis, to $X^\sharp$ as well.}
\end{remark}

\subsection{Power variation of the $\BSS$ process\emph{\ }and its
asymptotic theory}

Under Assumptions 1 and 2, \citet[][Theorem 3.1 and
equation (4.5)]{CHPP13} show that for $\upsilon \in \mathbb{N}$
\begin{equation}
\bar{V}\left( X;p,\upsilon\right)^n _{t} \overset{u.c.p.}{%
\rightarrow }V\left( X;p\right) _{t}=m_{p}\int_{0}^{t}\left\vert \sigma
_{s}\right\vert ^{p}ds,  \label{BN et al -Consistency}
\end{equation}%
where $m_{p}\equiv \mathbb{E}\left[ \left\vert U\right\vert ^{p}\right] ,$ $%
U\sim N\left( 0,1\right)$, and $\overset{u.c.p.}{\rightarrow }$ denotes uniform convergence in
$\mathbb{P}$-probability on compact sets. \citet[][Theorems 3.2 and
4.5]{CHPP13} also derive a joint asymptotic distribution of the vector $\Delta
_{n}^{-1/2}\left( \bar{V}\left( X;p,1\right)^n_{t}-V\left(
X;p\right)_t ,\bar{V}\left( X;p,2\right)^n_{t}-V\left( X;p\right)_t
\right)$. In particular,
\mikko{under Assumptions 1--3 \citep[Theorem 3.2]{CHPP13}},
\begin{equation}
\Delta _{n}^{-1/2}\left(
\begin{array}{c}
\bar{V}\left( X;p,1\right)^n_{t}-m_{p}\int_{0}^{t}\left\vert
\sigma _{s}\right\vert ^{p}ds \\
\bar{V}\left( X;p,2\right)^n_{t}-m_{p}\int_{0}^{t}\left\vert
\sigma _{s}\right\vert ^{p}ds%
\end{array}%
\right) \overset{st}{\rightarrow }N\left( 0,\Sigma_{p,t}\right) ,
\label{CLT-Joint}
\end{equation}%
where $\overset{st}{\rightarrow }$ denotes stable convergence, and
\begin{equation*}
\Sigma_{p,t}\equiv \Lambda _{p}\int_{0}^{t}\left\vert
\sigma _{s}\right\vert ^{2p}ds,  
\end{equation*}%
with the matrix $\Lambda _{p}=\left( \lambda _{p}^{ij}\right) _{1\leq
i,j\leq 2}$ given by%
\begin{eqnarray*}
\lambda _{p}^{11} &=&\underset{n\rightarrow \infty }{\lim }\Delta
_{n}^{-1}var\left( \bar{V}\left( B^{H};p,1\right)^n_{1}\right)
,\quad \lambda _{p}^{22}=\underset{n\rightarrow \infty }{\lim }\Delta
_{n}^{-1}var\left( \bar{V}\left( B^{H};p,2\right)^n_{1}\right) ,
\notag \\
\lambda _{p}^{12} &=&\underset{n\rightarrow \infty }{\lim }\Delta
_{n}^{-1}cov\left( \bar{V}\left( B^{H};p,1\right)^n_{1},\bar{V}%
\left( B^{H};p,2\right)^n _{1}\right) ,  \label{Lambda-ij}
\end{eqnarray*}%
with $B^{H}$ being a fractional Brownian motion with Hurst parameter $%
H=\alpha +1/2.$\footnote{Expressions for $\lambda_2^{ij}$ can be found in Appendix \ref{app:derive}.} Note that the computation of
the statistic $\bar{V}\left( X;p,\upsilon\right)^n _{t}$
requires knowledge of the factor $\tau _{n}\left( \upsilon \right) $, which is infeasible since it depends, among other things, on
the \mikko{roughness parameter} $\alpha $ of the $\BSS$\textit{\ }process $X$.\footnote{This approach \mikko{can be} made feasible by first estimating the factor $\tau_n(\upsilon)$, see \citet[][Appendix B]{BNPS14}. \mikko{However, this procedure has the shortcoming that the central limit theorem no longer holds.}}
Based on (\ref{BN et al -Consistency}) and (\ref{CLT-Joint}) \cite{CHPP13} construct consistent and asymptotically normal estimators of the
\mikko{roughness parameter} $\alpha .$ Under Assumptions 1 and 2, \citet[][equations 4.2 and 4.5]{CHPP13} show that%
\begin{equation}
\widehat{\alpha }\left( p\right)_t^n =h_{p}\left( COF\left(
p\right)^n_{t}\right) \overset{u.c.p.}{\rightarrow }\alpha,\label{alpha-hat-Consistency}
\end{equation}%
where%
\begin{equation}
h_{p}\left( x\right) =\frac{\log _{2}\left( x\right) }{p}-\frac{1}{2},\quad
x>0,  \label{function-h}
\end{equation}%
with $\log _{2}$ standing for the base-2 logarithm, whereas
\begin{equation}
COF\left( p\right)^n_{t}=\frac{V\left( X;p,2\right)^n_{t}}{V\left( X;p,1\right)^n_{t}}.
\label{function-COF}
\end{equation}%
By the delta method and the properties of stable convergence, \citet[][Propositions 4.2 and 4.6]{CHPP13} deduce a feasible CLT for the \mikko{roughness parameter} $\alpha .$ Assume that the conditions
of the CLT result (\ref{CLT-Joint}), with $\alpha \in \left( -\frac{1}{2}%
,0\right) \cup \left( 0,\frac{1}{4}\right) $ then for any $p\geq 2$, we have%
\begin{equation}
\frac{{p\log }\left( 2\right) V\left( X;p,1\right)^n_{t}\left(
\widehat{\alpha }\left( p\right)_t^n -\alpha \right) }{\sqrt{%
m_{2p}^{-1}V\left( X;2p,1\right)^n_{t}\left( -1,1\right)
\Lambda _{p}\left( -1,1\right) ^{T}}}\overset{d}{\rightarrow }N\left(
0,1\right) .  \label{CLT-Alpha-Test}
\end{equation}%

As mentioned in the introduction, the CLT \eqref{CLT-Alpha-Test} \mikko{may break down when $\alpha \in \left[\frac{1}{4},\frac{1}{2}\right)$. This motivated \cite{CHPP13} to develop a modified estimator implementing gaps between increments, from which the power variation \eqref{Power-Var} is computed. By letting the gaps widen sufficiently fast, the estimator satisfies a CLT and the relevant increments become asymptotically independent. 
In this case, one can develop a bootstrap method based on the idea of wild bootstrap. While we have also worked out the details of this approach, we relegate them to a web appendix \citep{BHLP16} for two reasons: Firstly,} the case \mikko{$\alpha \in \left[\frac{1}{4},\frac{1}{2}\right)$} seems to be of \mikko{limited} practical interest. Secondly, the theoretical results below show that the local fractional bootstrap method developed in this paper is valid for the whole range $\alpha \in \left(-\frac{1}{2},\frac{1}{2}\right)$.

\mikko{The results of \cite{CHPP13} do not explicitly allow for $\alpha = 0$. However, we can show} that under slightly \mikko{amended} assumptions\mikko{,} the LLN and CLT developed for the COF estimator \mikko{remain valid also} in this case. Indeed, only \mikko{Assumptions} 1(ii) and 2(ii) need to be changed. In the rest of this paper, when $\alpha = 0$, we thus \mikko{work under Assumptions 1--3} above with the following \mikko{modifications} to 1(ii) and 2(ii):

\begin{description}
\item[Assumption 1.]
\item[(ii')] $g^{\left( 2\right) }\left( x\right) = L_{g^{\left( 2\right)
}}\left( x\right) $ and for any $\epsilon >0,$ we have $g^{\left( 2\right)
}\in L^{2}\left( \left( \epsilon ,\infty \right) \right) .$ Furthermore, $%
\left\vert g^{\left( 2\right) }\right\vert $ is non-increasing on the
interval $\left( a,\infty \right) $ for some $a>0.$
\end{description}

\newpage

\begin{description}
\item[Assumption 2.]
\item[(ii')] $R^{\left(4\right) }\left( x\right) = f(x) L_{R^{\left( 4\right) }}\left( x\right) $, where the function $f$ is such that $\mikko{|}f(x)\mikko{|} \leq Cx^{-\beta}$ for some constants $C > 0$ and $\beta > 1/2$.
\end{description}

We now \mikko{obtain} the following result, which is proved in \mikko{Appendix \ref{sec:proofs}}.
\begin{proposition}\label{prop:a0}
Suppose Assumptions 1\mikko{--}3 hold. Then the LLN \eqref{alpha-hat-Consistency} and CLT \eqref{CLT-Alpha-Test} hold with $\alpha = 0$.
\end{proposition}
\begin{example}
The Ornstein-Uhlenbeck kernel $g(x) = e^{-\lambda x}$, $\lambda > 0$,
satisfies Assumptions 1 and 2 with $\alpha = 0$. Indeed, Assumption 1 is trivially seen to
hold and since
\begin{equation*}
R(x) = \lambda^{-1}\left( 1- e^{-\lambda x}\right) = xL_R(x),
\end{equation*}
where
\begin{equation*}
L_R(x) = x^{-1}\lambda^{-1}\left(1-e^{-\lambda x}\right)
\end{equation*}
is a slowly varying function, Assumption 2(i) also holds. We also have
\begin{equation*}
R^{(m)}(x) = (-1)^{m-1} \lambda^{m-1} e^{-\lambda x}, \quad m\geq 1,
\end{equation*}
so Assumption 2(ii') holds with $f(x) = e^{-\lambda x}$ and $L_{R^{\left( 4\right)}} = - \lambda^{-3}$. Lastly, 
\begin{equation*}
\lim_{x\downarrow 0}L_R(x) = 1,
\end{equation*}
so Assumption 2(iii) clearly also holds.
\end{example}

\section{The local fractional bootstrap} 
\label{sec:fracBoot}
\mikko{In this section, we introduce} a bootstrap method for a general class
of nonlinear transformations of the vector $\left( \bar{V}\left(
X;p,1\right)^n_{t},\bar{V}\left( X;p,2\right)^n_{t}\right)$. We then \mikko{use the method to approximate the sampling} distribution of the \mikko{roughness parameter} estimator $%
\widehat{\alpha }\left( p\right)_t^n $. In particular, we consider \mikko{a hypothesis test}
where the null hypothesis is
\begin{equation*}
 H_{0}:\alpha =\alpha _{0} 
\end{equation*}
\mikko{for some $\alpha
_{0}\in \left( -\frac{1}{2},\frac{1}{2}\right)$,} whereas the
alternative hypothesis is
\begin{equation*}
H_{1}:\alpha \neq \alpha _{0}\mikko{.}
\end{equation*}
Our \mikko{idea} is
to resample the high frequency second order differences of the
$\BSS$ process $X$ as defined in \eqref{Power-Var}. To be valid, the \mikko{method should} mimic the
dependence properties of \mikko{the increments of $X$.} As pointed out in \cite{CHPP13}, under Assumption 2, the 
\mikko{short-term behavior of the Gaussian core $G$ is similar to that of a fractional Brownian motion $B^H$ with Hurst parameter $H=\alpha +1/2$. More precisely, for any $t_0 \in \R$,
\begin{equation*}
\left(\frac{G_{\varepsilon t+t_0} - G_{t_0}}{\sqrt{Var(G_{\varepsilon}-G_0)}}\right)_{t \geq 0} \stackrel{d}{\rightarrow} (B^H_t)_{t \geq 0} \quad \textrm{in $C(\R^+)$}
\end{equation*}
as $\varepsilon \rightarrow 0$.
}

\mikko{Consider for now the
following constant-volatility toy model,
\begin{equation}
\tilde{X}_{t}=\sigma G_{t}=\sigma \int_{-\infty }^{t}g\left( t-s\right) dW_s, \quad t \geq 0,   \label{Toy-Model}
\end{equation}
obtained from $X$ by setting $\sigma _{t}=\sigma >0$ for all $t\in \R$.}
The above discussion suggests that in the context of $\tilde{X}$, the bootstrap scheme should be able to replicate \mikko{the correlation} structure of \mikko{the increments of} a fractional Brownian motion with Hurst
parameter $H=\alpha +1/2.$ \mikko{To this end, we propose the following local
fractional bootstrap algorithm:}
\begin{description}
\item[Step 1.] Specify a null \mikko{hypothesis $H_{0}:\alpha =\alpha
_{0}$ by fixing $\alpha_0 \in \left( -\frac{1}{2},\frac{1}{2}\right)$.}

\item[Step 2.] Generate $\left\lfloor t/\Delta _{n}\right\rfloor $ random
variables, $B_{\Delta_n}^{H},\ldots ,B_{\left\lfloor t/\Delta _{n}\right\rfloor }^{H}$, \mikko{which are independent of the original process $X$,} where $B^{H}$ is a fractional Brownian motion with Hurst parameter $H=\alpha_0 +1/2$.

\item[Step 3.] Finally, \mikko{return the observations}
\begin{equation}
X^{\ast }_{i\Delta_n}=\widehat{\sigma}\cdot B^{H}_{i\Delta_n},\quad i=2\upsilon ,\ldots
,\left\lfloor t/\Delta _{n}\right\rfloor,  \label{New-Boot-DGP}
\end{equation}%
where $\widehat{\sigma} = \widehat{\sigma }(p,\upsilon)^n$ is a consistent estimator of
the volatility $\sigma .$
\end{description}

This bootstrap algorithm deserves a few comments. First, we generate the
bootstrap observations under the null hypothesis\ $H_{0}:\alpha =\alpha
_{0}$; this feature is not only natural, but it is important to minimize
the probability of a type I error, see e.g. \cite{DM99}.
Second, although (\ref{New-Boot-DGP}) is motivated by the very simple model (%
\ref{Toy-Model}), as we will \mikko{show} below, this does not prevent the
bootstrap method to be valid more generally. In particular, its validity
extends to the case where the volatility is not constant as in (\ref{one}).
The choice of $\widehat{\sigma}$ may change depending
on the statistics of interest. As we will see shortly, for instance, when we
consider the vector $\Delta _{n}^{-1/2}\left( \bar{V}\left( X;p,1\right)^n_{t},\bar{V}\left( X;p,2\right)^n_{t}\right) $, one
could simply use $\widehat{\sigma } = \widehat{\sigma }(p,\upsilon)_t^n =\left( m_{p}^{-1}%
\bar{V}\left( X;p,\upsilon\right)^n_{t}\right)^{1/p}$, cf. Equation \eqref{BN et al -Consistency}.

Define the bootstrap power variations analogues of (\ref{Power-Var}) and (%
\ref{Power-Var-Bar}), respectively, as follows
\begin{eqnarray}
V^{\ast }\left( X,B^{H};p,\upsilon\right)^n_{t}&\equiv& \frac{%
\left\vert  \widehat{\sigma }(p,\upsilon)_t^n\right\vert ^{p}}{\overline{\mu}(p,\upsilon)_t^n} V(B^H;p,\upsilon)^n_t,  \label{Boot-New-Power-Var} \\
\bar{V}^{\ast }\left( X,B^{H};p,\upsilon\right)^n_{t} &\equiv
&\Delta _{n}\tau _{n}\left( \upsilon \right) ^{-p}V^{\ast }\left(
X,B^{H};p,\upsilon\right)^n_{t}  \notag \\
&=&\frac{\left\vert  \widehat{\sigma }(p,\upsilon)_t^n\right\vert ^{p}}{%
\overline{\mu}(p,\upsilon)_t^n}\bar{V}\left( B^{H};p,\upsilon\right)^n_{t},  \label{Boot-New-Power-Var-Bar}
\end{eqnarray}%
where $\overline{\mu}(p,\upsilon)_t^n =\Delta _{n}\tau _{n}\left(
\upsilon \right) ^{-p}\mu(p,\upsilon)_t^n$ with $\mu(p,\upsilon)_t^n =\mathbb{E}^{\ast }\left( V\left( B^{H};p,\upsilon
\right)^n_{t}\right) $.

\begin{lemma}
\label{LemmaBoot-New-moments} Consider (\ref{one}), (\ref{Boot-New-Power-Var}%
), and (\ref{Boot-New-Power-Var-Bar}) where $B^{H}$ is a fractional Brownian
motion with Hurst parameter $H=\alpha _{0}+1/2$. It follows that

\begin{description}
\item[(i)] $\mathbb{E}^{\ast }\left( \bar{V}^{\ast }\left(
X,B^{H};p,\upsilon\right)^n_{t}\right) =\left\vert \widehat{\sigma }(p,\upsilon)_t^n \right\vert ^{p}.$

\item[(ii)] $Var^{\ast }\left( \Delta _{n}^{-1/2}\bar{V}^{\ast }\left(
X,B^{H};p,1\right)^n_{t}\right) =\underset{\equiv \lambda
_{p,n}^{11}}{\underbrace{\Delta _{n}^{-1}Var\left( \bar{V}\left(
B^{H};p,1\right)^n_{t}\right) }}  \frac{\left\vert\widehat{\sigma }(p,1)_t^n\right\vert ^{2p}}{\left( \overline{\mu}(p,1)_t^n\right)^{2}},$

\item[(iii)] $Var^{\ast }\left( \Delta _{n}^{-1/2}\bar{V}^{\ast }\left(
X,B^{H};p,2\right)^n_{t}\right) =\underset{\equiv \lambda
_{p,n}^{22}}{\underbrace{\Delta _{n}^{-1}Var\left( \bar{V}\left(
B^{H};p,2\right)^n_{t}\right) }} \frac{\left\vert\widehat{\sigma }(p,2)_t^n\right\vert ^{2p}}{\left( \overline{\mu}(p,2)_t^n\right)^{2}},$

\item[(iv)] $Cov^{\ast }\left( \Delta _{n}^{-1/2}\bar{V}^{\ast }\left(
X,B^{H};p,1\right)^n_{t},\Delta _{n}^{-1/2}\bar{V}^{\ast
}\left( X,B^{H};p,2\right)^n_{t}\right) $

$=\underset{\equiv \lambda _{p,n}^{12}}{\underbrace{\Delta
_{n}^{-1}Cov\left( \bar{V}\left( B^{H};p,1\right)^n_{t},\bar{V}%
\left( B^{H};p,2\right)^n_{t}\right) }} \frac{\left\vert\widehat{\sigma }(p,1)_t^n\right\vert ^{p}\left\vert\widehat{\sigma }(p,2)_t^n\right\vert ^{p}}{ \overline{\mu}(p,1)_t^n \overline{\mu}(p,2)_t^n},$

\item[(v)] If $\left\vert \widehat{\sigma }(p,\upsilon)_t^n\right\vert
^{2p}\overset{u.c.p.}{\rightarrow }\int_{0}^{t}\left\vert \sigma
_{s}\right\vert ^{2p}ds$ and $\left\vert \widehat{\sigma }(p,1)_t^n \right\vert^{p}\left\vert \widehat{\sigma }(p,2)_t^n\right%
\vert ^{p}\overset{u.c.p.}{\rightarrow }\int_{0}^{t}\left\vert \sigma
_{s}\right\vert ^{2p}ds$, \text{ then }
\begin{equation*}
p\underset{n\rightarrow \infty }{\lim }\Sigma^{\ast }\left(
X,B^{H};p\right)^n_{t}-\widetilde{\Sigma }_{p,t}^n=0,  
\end{equation*}%
where%
\begin{eqnarray*}
\Sigma^{\ast }\left( X,B^{H};p\right)^n_{t} &\equiv
&Var^{\ast }\left( \Delta _{n}^{-1/2}\left(
\begin{array}{c}
\bar{V}^{\ast }\left( X,B^{H};p,1\right)^n_{t} \\
\bar{V}^{\ast }\left( X,B^{H};p,2\right)^n_{t}%
\end{array}%
\right) \right), \text{ and } \\
\widetilde{\Sigma }_{p,t}^n &\equiv &\Lambda_{p,t}^n\int_{0}^{t}\left\vert \sigma _{s}\right\vert ^{2p}ds,
\end{eqnarray*}%
such that
\begin{equation*}
\Lambda_{p,t}^n=\left(
\begin{array}{cc}
\left(\overline{\mu}(p,1)_t^n\right) ^{-2}\lambda _{p,n}^{11} & \left(
\overline{\mu}(p,1)_t^n\right) ^{-1}\left( \overline{\mu}(p,2)_t^n\right) ^{-1}\lambda _{p,n}^{12} \\
\left(\overline{\mu}(p,2)_t^n\right) ^{-1}\left( \overline{\mu}(p,2)_t^n\right) ^{-1}\lambda _{p,n}^{12} & \left( \overline{\mu}(p,2)_t^n\right) ^{-2}\lambda _{p,n}^{22}%
\end{array}%
\right) . 
\end{equation*}
\end{description}
\end{lemma}

Part (v) of Lemma \ref{LemmaBoot-New-moments} shows that the bootstrap
variance $\Sigma^{\ast }\left( X,B^{H};p\right)^n_{t}$
will only be a consistent estimator of $\Sigma_{p,t}$
under the general model (\ref{one}) if the following three conditions hold
true:%
\begin{equation}
\left\vert \widehat{\sigma }(p,\upsilon)_t^n\right\vert ^{2p}\overset{%
u.c.p.}{\rightarrow }\int_{0}^{t}\left\vert \sigma _{s}\right\vert ^{2p}ds, \quad \left\vert \widehat{\sigma }(p,1)_t^n\right\vert ^{p}\left\vert
\widehat{\sigma }(p,2)_t^n\right\vert ^{p}\overset{u.c.p.}{\rightarrow }%
\int_{0}^{t}\left\vert \sigma _{s}\right\vert ^{2p}ds
\label{condition1}
\end{equation}%
and
\begin{equation}
\Lambda_{p,t}^n\rightarrow \Lambda _{p} \quad \text{(i.e. }%
\overline{\mu}(p,\upsilon)_t^n\rightarrow 1\text{)}.  \label{condition2}
\end{equation}%
It is easy to see that letting $\widehat{\sigma}(p,\upsilon)_t^n=\left(
m_{2p}^{-1}\bar{V}\left( X;2p,\upsilon\right)^n_{t}\right)
^{1/2p}$ will satisfy (\ref{condition1}). However, it may not be possible to
satisfy (\ref{condition2}). For instance, for $p=2$ (which is \mikko{the most} important
case in practice), one can show that (see Appendix \ref{app:derive})%
\begin{equation*}
\mu(2,1)_t^n=\left( \left\lfloor t/\Delta _{n}\right\rfloor -1\right)
\Delta _{n}^{2H}\left( 4-2^{2H}\right).
\label{Moment-Frac-Brownian}
\end{equation*}%
However, despite $\Sigma^{\ast }\left( X,B^{H};p\right)^n_{t}$ not being consistent for $\Sigma_{p,t}$, we can
still achieve an asymptotically valid bootstrap for the studentized
distribution. To this end, we need to find a consistent estimator of $\Sigma
^{\ast }\left( X,B^{H};p\right)^n_{t}$ based on bootstrap
observations. In the following, without loss of generality, we will use both
choices of $\widehat{\sigma}(p,\upsilon)_t^n$ given by
\begin{equation}
\widehat{\sigma}(p,\upsilon)_t^n=\left( m_{p}^{-1}\bar{V}\left(
X;p,\upsilon\right)^n_{t}\right) ^{1/p},\text{ for }%
p=1,2,\ldots ,\text{ and }\upsilon =1,2  \label{Sigma-hat-1}
\end{equation}%
and%
\begin{equation}
\widehat{\sigma}(p,\upsilon)_t^n=\left( m_{\beta p}^{-1}\bar{V}\left(
X;\beta p,\upsilon\right)^n_{t}\right) ^{1/\beta p},\text{ for
}\beta >0,\text{ }p=1,2,\ldots ,\text{ and }\upsilon =1,2.
\label{Sigma-hat-2}
\end{equation}%
We propose the following consistent estimator of $\Sigma^{\ast }\left(
X,B^{H};p\right)^n_{t}$ defined by%
\begin{equation}
\widehat{\Sigma }^{\ast }\left( X,B^{H};p\right)^n_{t} = \left(
\begin{array}{cc}
\lambda _{p,n}^{11} \frac{\left\vert \widehat{\sigma}(p,1)_t^{n\ast}\right\vert ^{2p}}{\left( \overline{\mu}(p,1)_t^n\right) ^{2}}   & \lambda
_{p,n}^{12}  \frac{\left\vert \widehat{\sigma}(p,1)_t^{n\ast}\right\vert
^{p}\left\vert \widehat{\sigma}(p,2)_t^{n\ast}\right\vert ^{p}}{\overline{\mu}(p,1)_t^n\overline{\mu}(p,2)_t^n}   \\
 \lambda _{p,n}^{12} \frac{  \left\vert \widehat{\sigma}(p,1)_t^{n\ast}\right\vert ^{p}\left\vert \widehat{\sigma}(p,2)_t^{n\ast}\right\vert
^{p}}{\overline{\mu}(p,1)_t^n\overline{\mu}(p,2)_t^n}  &  \lambda
_{p,n}^{22} \frac{\left\vert \widehat{\sigma}(p,2)_t^{n\ast}\right\vert ^{2p}}{ \left( \overline{\mu}(p,2)_t^n\right) ^{2}}
\end{array}%
\right)   \label{Boot-New-Assymp-Var-hat}
\end{equation}%
where%
\begin{equation}
\left\vert \widehat{\sigma}(p,\upsilon)_t^{n\ast}\right\vert ^{p}=\bar{V}%
^{\ast }\left( X,B^{H};p,\upsilon\right)^n_{t}.
\label{Sigma-hat-1-star}
\end{equation}

\begin{theorem}
\label{TheorJointCLT-a} Suppose \mikko{that Assumptions 1--3 hold} for $\alpha \in \left( -%
\frac{1}{2},\mikko{\frac{1}{2}}\right)$. Assume that bootstrap observations are given by (%
\ref{New-Boot-DGP}) where $B^{H}$ is a fractional Brownian motion with Hurst
parameter $H=\alpha _{0}+1/2$. It follows that as $n\rightarrow \infty ,$%
\begin{eqnarray*}
\widehat{\mathbf{S}}_{n}^{\ast } &=&\left( \widehat{\Sigma}^{\ast
}\left( X,B^{H};p\right)^n_{t}\right) ^{-1/2}\Delta
_{n} ^{-1/2}\left(
\begin{array}{c}
\bar{V}^{\ast }\left( X,B^{H};p,1\right)^n_{t}-\mathbb{E}^{\ast
}\left( \bar{V}^{\ast }\left( X,B^{H};p,1\right)^n_{t}\right)
\\
\bar{V}^{\ast }\left( X,B^{H};p,2\right)^n_{t}-\mathbb{E}^{\ast
}\left( \bar{V}^{\ast }\left( X,B^{H};p,2\right)^n_{t}\right)%
\end{array}%
\right) \\
&&\overset{d^{\ast }}{\rightarrow }N\left( 0,I_{2}\right) ,
\end{eqnarray*}%
in prob-$\mathbb{P}.$
\end{theorem}

Thus, we can deduce a bootstrap CLT result for the bootstrap smoothness
parameter estimator $\widehat{\alpha }^{\ast }\left( p\right)_t^n$ analogue of $\widehat{\alpha }\left( p\right)_t^n.$ To this end, let%
\begin{equation*}
\widehat{\alpha }^{\ast }\left( p\right)_t^n=h_{p}\left(
COF^{\ast }\left( p\right)_t^n\right) ,
\end{equation*}%
where $h_{p}\left( \cdot \right) $ is defined by (\ref{function-h}), whereas
\begin{equation}
COF^{\ast }\left( p\right)^n_{t}=\frac{V^{\ast }\left(
X,B^{H};p,2\right)^n_{t}}{V^{\ast }\left( X,B^{H};p,1\right)^n_{t}}.  \label{function-COF-Boot-a}
\end{equation}%
To understand the asymptotic behavior of $COF^{\ast }\left( p\right)_t^n,$ we can write
\begin{equation*}
COF^{\ast }\left( p\right)_t^n=\left( \frac{\tau _{n}\left(
2\right) ^{2}}{\tau _{n}\left(1\right) ^{2}}\right)
^{p/2}\frac{\bar{V}^{\ast }\left( X,B^{H};p,2\right)^n_{t}}{%
\bar{V}^{\ast }\left( X,B^{H};p,1\right)^n_{t}}.
\end{equation*}%
From Assumption 2, we have
\begin{equation*}
\left( \frac{\tau _{n}\left( 2\right) ^{2}}{\tau _{n}\left(
1\right) ^{2}}\right) ^{p/2}\rightarrow 2^{\frac{\left( 2\alpha
+1\right) p}{2}},
\end{equation*}%
and thus by Theorem \ref{TheorJointCLT-a}, part (i) of Lemma \ref%
{LemmaBoot-New-moments}\mikko{,} in conjunction with (\ref{Sigma-hat-1}) and (\ref%
{Sigma-hat-2})\mikko{,} and by using equation (\ref{BN et al -Consistency}) with $%
\upsilon =1,2,$ we can deduce that
\begin{equation*}
\frac{\bar{V}^{\ast }\left( X,B^{H};p,2\right)^n_{t}}{\bar{V}%
^{\ast }\left( X,B^{H};p,1\right)^n_{t}}\overset{\mathbb{P}%
^{\ast }}{\rightarrow }1,\quad \text{in prob-}\mathbb{P}.
\end{equation*}%
It follows that, by applying the delta method on the CLT results of Theorem %
\ref{TheorJointCLT-a}, we can characterize the distribution of $\widehat{\alpha
}^{\ast }\left( p\right)_t^n.$ These results are summarized
in the following theorem.

\begin{theorem}
\label{Theor-Alpha-hat-boot-a} Suppose \mikko{that Assumptions 1--3 hold} for $\alpha \in
\left( -\frac{1}{2},\mikko{\frac{1}{2}}\right)$. Assume that bootstrap observations
are given by (\ref{New-Boot-DGP}) where $B^{H}$ is a fractional Brownian
motion with Hurst parameter $H=\alpha _{0}+1/2$. It follows that for any $%
p\geq 2,$ as $n\rightarrow \infty ,$
\begin{equation*}
T_{\widehat{\alpha },n}^{\ast }\equiv \Delta _{n}^{-1/2}\frac{\left(
\widehat{\alpha }^{\ast }\left( p\right)_t^n-\widetilde{%
\alpha }\left( p\right)_t^n\right) }{\sqrt{\widehat{V}%
^{\ast }\left( \widehat{\alpha }\right) _{t}}}\overset{d^{\ast }}{%
\rightarrow }N\left( 0,1\right) ,\text{ in prob-}\mathbb{P},
\end{equation*}%
where%
\begin{equation}
\widetilde{\alpha }\left( p\right)_t^n=h_{p}\left(
\widetilde{COF}\left( p\right)^n_{t}\right) ,
\label{alpha-tilde}
\end{equation}%
such that
\begin{eqnarray*}
\widetilde{COF}\left( p\right)^n_{t} &=&\left( \frac{\tau
_{n}\left( 2\right) ^{2}}{\tau _{n}\left( 1\right) ^{2}}%
\right) ^{p/2}\frac{\left\vert \widehat{\sigma}(p,2)_t^n\right\vert ^{p}%
}{\left\vert \widehat{\sigma}(p,1)_t^n\right\vert ^{p}} \\
&=&\left\{
\begin{array}{lcl}
\frac{V\left( X;p,2\right)^n_{t}}{V\left( X;p,1\right)^n_{t}}, & \text{if} &\widehat{\sigma}(p,\upsilon)_t^n
\text{ is given by }(\ref{Sigma-hat-1}), \\
\left( \frac{V\left( X;\beta p,2\right)^n _{t}}{V\left( X;\beta
p,1\right)^n_{t}}\right) ^{1/\beta }, & \text{if} &\widehat{\sigma}(p,\upsilon)_t^n\ \text{ is given by }(\ref{Sigma-hat-2}),%
\end{array}%
\right.
\end{eqnarray*}%
and the estimator of the asymptotic variance $\widehat{V}^{\ast }\left(
\widehat{\alpha }\right) _{t}$ is defined as%
\begin{equation}
\widehat{V}^{\ast }\left( \widehat{\alpha }\right) _{t}=\frac{1}{\left( {%
p\log }\left( 2\right) \right) ^{2}}\widehat{\varsigma }^{\ast }\left(
X,B^{H};p\right)^n_{t},  \label{Var-Boot-Frac-Brownian-Hat}
\end{equation}%
with%
\begin{multline*}
\widehat{\varsigma }^{\ast }\left( X,B^{H};p\right)^n_{t}  \\ 
\begin{aligned}
& =\frac{\lambda _{p,n}^{11}    }{\left[ \mathbb{E}^{\ast
}\left( \bar{V}^{\ast }\left( X,B^{H};p,1\right)^n_{t}\right) %
\right] ^{2}}       \frac{\left( \left\vert \widehat{\sigma}(p,1)_t^{n\ast}\right\vert ^{p}\right) ^{2}}{\left(\overline{\mu}(p,1)_t^n\right) ^{2}}    +  \frac{\lambda
_{p,n}^{22}}{\left[ \mathbb{E}^{\ast }\left( \bar{V}^{\ast
}\left( X,B^{H};p,2\right)^n _{t}\right) \right] ^{2}}           \frac{\left( \left\vert \widehat{\sigma}(p,2)_t^{n\ast}\right\vert ^{p}\right) ^{2}}{\left( \overline{\mu}(p,2)_t^n\right) ^{2}}
 \\ 
& \qquad -2\frac{\lambda _{p,n}^{12}}{\mathbb{E}^{\ast }\left( \bar{V}^{\ast }\left( X,B^{H};p,1\right)^n_{t}\right) \mathbb{E}^{\ast }\left( \bar{V}^{\ast }\left(
X,B^{H};p,2\right)^n_{t}\right) }  \frac{\left\vert \widehat{\sigma}(p,1)_t^{n\ast}\right\vert ^{p}\left\vert \widehat{\sigma}(p,2)_t^{n\ast}\right\vert ^{p}}{\overline{\mu}(p,1)_t^n\overline{\mu}(p,2)_t^n}.
\end{aligned}
\end{multline*}
\end{theorem}

\begin{remark}\label{rem:bootFeasible}
Note that the bootstrap statistic $T_{\widehat{\alpha },n}^{\ast }$ is
feasible: it is only a function of the original sample of the observed data $%
\left\{X_{i\Delta_n}\right\} ,$ the fractional Brownian
motion generated in Step 2, $\left\{ B^H_{i\Delta_n}\right\}$, and their absolute moments $\left\{ E\left\vert B^H_{i\Delta_n} - 2B^H_{(i-\upsilon)\Delta_n} + B^H_{(i-2\upsilon)\Delta_n} \right\vert ^{p}\right\} $. To see this, write
\begin{equation*}
\widehat{\alpha }^{\ast }\left( p\right)_t^n=h_{p}\left(
COF^{\ast }\left( p\right)^n_{t}\right) ,
\end{equation*}%
where $h_{p}\left( \cdot \right) $ and $COF^{\ast }\left( p\right)^n_{t}$ are given in (\ref{function-h}) and (\ref%
{function-COF-Boot-a}), respectively. Given (\ref{function-COF-Boot-a}) and (%
\ref{Boot-New-Power-Var}), it follows that%
\begin{eqnarray}
COF^{\ast }\left( p\right)^n_{t} &=&\left( \frac{\tau
_{n}\left( 2\right) ^{2}}{\tau _{n}\left(1\right) ^{2}}%
\right) ^{p/2}\frac{\bar{V}^{\ast }\left( X,B^{H};p,2\right)^n_{t}}{\bar{V}^{\ast }\left( X,B^{H};p,1\right)^n_{t}}  \notag \\
&=&\left\{
\begin{array}{lcl}
\frac{\mu(p,1)_t^n}{\mu(p,2)_t^n}\frac{V\left( B^{H};p,2\right)^n_{t}}{V\left( B^{H};p,1\right)^n_{t}}\frac{V\left(
X;p,2\right)^n_{t}}{V\left( X;p,1\right)^n_{t}}, &
\text{if} & \widehat{\sigma}(p,\upsilon)_t^n \text{ is given by }(\ref%
{Sigma-hat-1}), \\
\frac{\mu(p,1)_t^n}{\mu(p,2)_t^n}\frac{V\left( B^{H};p,2\right)^n_{t}}{V\left( B^{H};p,1\right)^n_{t}}\left( \frac{%
V\left( X;\beta p,2\right)^n_{t}}{V\left( X;\beta p,1\right)^n_{t}}\right) ^{1/\beta }, & \text{if} &\widehat{\sigma}(p,\upsilon)_t^n \text{ is given by }(\ref{Sigma-hat-2}).%
\end{array}%
\right.  \label{function-COF-Boot-a-feasible}
\end{eqnarray}%
Similarly, when  $\widehat{\sigma}(p,\upsilon)_t^n$ is given by $(\ref%
{Sigma-hat-1})$ or $(\ref{Sigma-hat-2})$, we can write $\widehat{V}^{\ast }\left( \widehat{\alpha }\right) _{t}$ given in Theorem \ref{Theor-Alpha-hat-boot-a} through%
\begin{equation*}
\widehat{V}^{\ast }\left( \widehat{\alpha }\right) _{t}=\frac{1}{\left( {%
p\log }\left( 2\right) \right) ^{2}}\widehat{\varsigma }^{\ast }\left(
X,B^{H};p\right)^n_{t},
\end{equation*}%
with%
\begin{equation*}
\widehat{\varsigma }^{\ast }\left( X,B^{H};p\right)^n_{t}=A+B+C, 
\end{equation*}%
where
\begin{eqnarray*}
A&=&\Delta _{n}^{-1}\left( \mu(p,1)_t^n\right) ^{-4}\left( V\left(
B^{H};p,1\right)^n_{t}\right) ^{2}Var\left( V\left(
B^{H};p,1\right)^n_{t}\right), \\
B&=&\Delta _{n}^{-1}\left( \mu(p,2)_t^n\right) ^{-4}\left( V\left(
B^{H};p,2\right)^n_{t}\right) ^{2}Var\left( V\left(
B^{H};p,2\right)^n_{t}\right), \\
C&=&C_{1}\cdot C_{2},
\end{eqnarray*}%
and
\begin{eqnarray*}
C_{1} &=&-2\left( \mu(p,1)_t^n\right) ^{-2}\left(\mu(p,2)_t^n\right) ^{-2} V\left( B^{H};p,1\right)^n_{t} V\left( B^{H};p,2\right)^n_{t}, \\
C_{2} &=&\Delta _{n}^{-1}Cov\left( V\left( B^{H};p,1\right)^n_{t},V\left( B^{H};p,2\right)^n_{t}\right) .
\end{eqnarray*}
Expressions for $Var\left( V\left(B^{H};p,\upsilon \right)^n_{t}\right)$, $Cov\left( V\left( B^{H};p,1\right)^n_{t},V\left( B^{H};p,2\right)^n_{t}\right)$ and $\mu(p,\upsilon)$ for $p=2$ can be found in Appendix \ref{app:derive}.

\end{remark}

\subsection{Bootstrap implementation}

We can use the bootstrap method proposed above to test \mikko{hypotheses on the roughness of} the \mikko{sample} paths of a $\BSS$ process. Consider the following, where the null hypothesis is\ $H_{0}:\alpha =\alpha _{0}$ \mikko{for some $\alpha _{0}\in \left(- \frac{1}{2},\frac{1}{2}\right)$,}  whereas the alternative hypothesis is $H_{1}:\alpha \neq \alpha _{0}$. We let $p=2$. For a given time period $\left[ 0,t\right] $ with step size $\Delta _{n}=\frac{t}{n}$ we suppose we have $n+1\in \mathbb{N}$ observations $\mikko{\mathbb{X}} = (X_0, X_{\Delta_n},\ldots,X_{n\Delta_n})$ of a $\BSS$ process. Below, $B$ is the number of bootstrap replications (e.g. $B=999$).

\begin{description}
\item[Algorithm for \mikko{hypothesis testing using} the Local Fractional Bootstrap]
\end{description}

\begin{description}
\item[1.] From the data $\mikko{\mathbb{X}}$, compute the estimate of the \mikko{roughness parameter} $\alpha$ given by
\begin{equation*}
\widehat{\alpha }\left(2\right)_t^n=h_{2}\left( COF\left(
2\right)^n_{t}\right),
\end{equation*}%
where $h_{p}\left( \cdot \right) $ and $COF\left( 2\right)^n_{t}$
are given in (\ref{function-h}) and (\ref{function-COF}), respectively.
Then, compute an estimator of the asymptotic variance $V\left( \widehat{%
\alpha }\right)^n_{t}=\underset{n\rightarrow \infty }{\lim }Var\left(
\widehat{\alpha }\left( 2\right)^n_{t}\right) ,$ given by%
\begin{equation*}
\widehat{V}\left( \widehat{\alpha }\right) _{t}=n\frac{%
m_{2p}^{-1}V\left( X;4,1\right)^n_{t}\left( -1,1\right) \Lambda
_{2}\left( -1,1\right) ^{T}}{\left( {2\log }\left( 2\right) V\left(
X;2,1\right)^n_{t}\right) ^{2}}.
\end{equation*}

\item[2.] Simulate $n+1$ \mikko{observations} $B_{0}^{H},B_{\Delta_n}^{H},\ldots ,B_{n\Delta_n}^{H}$ of a
fractional Brownian motion with Hurst parameter $H=\alpha _{0}+1/2$ \mikko{that are independent of the data $\mikko{\mathbb{X}}$.}

\item[3.] Using the simulated sample $(B_{0}^{H},B_{1}^{H},\ldots ,B_{n}^{H})$, compute the estimate of the bootstrap \mikko{roughness parameter} $\widehat{%
\alpha }^{\ast }\left( 2\right)^n_{t},$ given by
\begin{equation*}
\widehat{\alpha }^{\ast }\left( 2\right)^n_{t}=h_{2}\left(
COF^{\ast }\left( 2\right)^n_{t}\right) ,
\end{equation*}%
where $h_{p}\left( \cdot \right) $ and $COF^{\ast }\left( p\right)^n_{t}$ are given in (\ref{function-h}) and (\ref%
{function-COF-Boot-a-feasible}), respectively.

\item[4.] \mikko{The actual test relies} on the
bootstrap studentized statistic. Thus, compute%
\begin{equation*}
\text{ }T_{\widehat{\alpha },n}^{\ast }=\Delta _{n}^{-1/2}\frac{\left(
\widehat{\alpha }^{\ast }\left( 2\right)^n_{t}-\widetilde{%
\alpha }\left( 2\right)^n_{t}\right) }{\sqrt{\widehat{V}%
^{\ast }\left( \widehat{\alpha }\right) _{t}}},
\end{equation*}%
where $\widetilde{\alpha }\left( 2\right)^n_{t}$ is given by (%
\ref{alpha-tilde}), whereas $\widehat{\alpha }^{\ast }\left( 2\right)^n_{t}$ is obtained in step 3, and $\widehat{V}^{\ast }\left(
\widehat{\alpha }\right) _{t}$ is defined in (\ref%
{Var-Boot-Frac-Brownian-Hat}).

\item[5.] Repeat steps \mikko{2--4} $B$ times and \mikko{store} the values of $T_{\widehat{%
\alpha },n,j}^{\ast }$, $j=1,\ldots ,B.$

\item[6.] Reject $H_{0}:\alpha =\alpha _{0},$ when
\begin{equation*}
\alpha _{0}\notin IC_{perc-t,1-\gamma }^{\ast }=\left[ \widehat{\alpha }%
\left( 2\right)^n_{t}-n^{-1/2}q_{1-\gamma /2}^{\ast }\sqrt{%
\widehat{V}\left( \widehat{\alpha }\right) _{t}},\widehat{\alpha }%
\left( 2\right)^n_{t}-n^{-1/2}q_{\gamma /2}^{\ast }\sqrt{%
\widehat{V}\left( \widehat{\alpha }\right) _{t}}\right] ,
\end{equation*}%
where $q_{\gamma /2}^{\ast }$ and $q_{1-\gamma /2}^{\ast }$ are the $\gamma
/2$ and $1-\gamma /2$ quantiles of the bootstrap distribution of $%
T_{n}^{\ast },$ respectively.
\end{description}

\section{Monte Carlo simulation study}

\label{sec:MC}

In this section, we evaluate the finite sample performance of the \mikko{test based on} local fractional bootstrap and compare \mikko{it to the performance of the CLT-based test}. In our simulations we take $g$ to be the gamma kernel, i.e. $%
g(x)=x^{\alpha }e^{-\lambda x}$ for $x>0$ and with $\lambda =1$. For an
in-depth analysis of the \mikko{theoretical properties of the} gamma kernel, see \cite{BN12,BN16a}.  We consider $\alpha  \in \{-1/3, -1/6, 0, 1/6, 1/3\}$\mikko{; recall that the CLT for the COF estimator does not hold when when $\alpha \in [1/4,1/2)$.} 
We experimented with \mikko{several} different values of $\lambda $ and $\alpha$, but the results were \mikko{in all cases} very similar to what we find below. We consider three specifications for the stochastic volatility process:
\mikko{\begin{itemize}
\item constant volatility (NoSV),
\item one-factor stochastic volatility (SV1F),
\item two-factor stochastic volatility (SV2F).
\end{itemize}}
The details of these specifications, along with the simulation procedure, are explained in Appendix \ref{sec:sim}. Our
investigations concern the finite sample size of the test $H_{0}:\alpha =\alpha _{0}$
against the \mikko{two}-sided alternative $H_{1}:\alpha \neq \alpha _{0}$ \mikko{at
(nominal) $5\%$ level}. We also calculate
the \emph{size-adjusted power} of the test \mikko{at $5\%$ level} where the \mikko{(now false)} null hypothesis is $H_{0}:\alpha =0$.\footnote{The size-adjusted power was obtained in the following way: Using Monte Carlo simulations we found critical values that would result in the CLT obtaining the same size as the bootstrap (see Table \ref{tab:size} for these size numbers); these critical values were then used to determine the power of the CLT-based test. As the bootstrap test seems to be correctly sized for all $n$ this was not size-corrected. By using the actual size of the bootstrap test (instead of the nominal $5\%$) to calculate the size-adjusted critical value for the CLT-based test, the power properties of the two tests become comparable.}
\mikko{The null hypothesis $H_{0}:\alpha =0$}
is particularly interesting as \mikko{$\alpha =0$} is a necessary condition for \mikko{the semimartingality of $X$}. Further, in the case of the gamma kernel $g(x)=x^{\alpha }e^{-\lambda x}$ that we consider here, $\alpha =0$ implies that the $\BSS$ process actually is an Ornstein-Uhlenbeck process.

Tables \ref{tab:size} and \ref{tab:power} contain the results of our Monte Carlo study and detail the finite sample properties of both the CLT and the local fractional bootstrap. Table \ref{tab:size} \mikko{presents} rejection rates of $H_0$ when $H_0$ is true (i.e. the size), while Table \ref{tab:power} \mikko{displays} rejection rates of $H_0$ when $H_1$ is true (i.e. the power). Some clear conclusions \mikko{can be drawn}. \mikko{Firstly}, the bootstrap method offers clear gains in the size of the test when the number of observations, $n$, is small. Secondly, the power of the CLT is slightly better for $\alpha < 0$, while the opposite is true for $\alpha \geq 0$. Finally, the presence or absence of SV does not alter results very much, except in the case of SV2F where the methods lose some power.

\begin{table}
\caption{\it Rejection rates under $H_0$}
\begin{center}
\begin{tabularx}{.95\textwidth}{@{\extracolsep{\stretch{1}}}lcccccccccc@{}} 
\multicolumn{11}{l}{Panel A: NoSV} \\
\toprule
$n$  & \multicolumn{2}{c}{$\alpha = -1/3$} & \multicolumn{2}{c}{$\alpha = -1/6$} & \multicolumn{2}{c}{$\alpha = 0$} & \multicolumn{2}{c}{$\alpha = 1/6$} & \multicolumn{2}{c}{$\alpha = 1/3$} \\ 
\cmidrule{2-11}
 & CLT & boot & CLT & boot & CLT & boot & CLT & boot & CLT & boot \\ 
 $ 20 $ & $ 0.0968 $ & $ 0.0470 $ & $ 0.0950 $ & $ 0.0454 $ & $ 0.0968 $ & $ 0.0354 $ & $ 0.1044 $ & $ 0.0478 $ & $ 0.1110 $ & $ 0.0456 $\\
 $ 40 $ & $ 0.0742 $ & $ 0.0534 $ & $ 0.0728 $ & $ 0.0540 $ & $ 0.0754 $ & $ 0.0488 $ & $ 0.0746 $ & $ 0.0584 $ & $ 0.0852 $ & $ 0.0558 $\\
 $ 80 $ & $ 0.0642 $ & $ 0.0568 $ & $ 0.0562 $ & $ 0.0512 $ & $ 0.0644 $ & $ 0.0550 $ & $ 0.0638 $ & $ 0.0526 $ & $ 0.0726 $ & $ 0.0610 $\\
 $ 160 $ & $ 0.0620 $ & $ 0.0596 $ & $ 0.0638 $ & $ 0.0598 $ & $ 0.0536 $ & $ 0.0514 $ & $ 0.0576 $ & $ 0.0558 $ & $ 0.0572 $ & $ 0.0528 $\\
 $ 320 $ & $ 0.0568 $ & $ 0.0556 $ & $ 0.0540 $ & $ 0.0530 $ & $ 0.0562 $ & $ 0.0526 $ & $ 0.0610 $ & $ 0.0582 $ & $ 0.0548 $ & $ 0.0516 $\\
\bottomrule 
\end{tabularx}
\begin{tabularx}{.95\textwidth}{@{\extracolsep{\stretch{1}}}lcccccccccc@{}} 
\multicolumn{11}{l}{Panel B: SV1F} \\
\toprule
$n$  & \multicolumn{2}{c}{$\alpha = -1/3$} & \multicolumn{2}{c}{$\alpha = -1/6$} & \multicolumn{2}{c}{$\alpha = 0$} & \multicolumn{2}{c}{$\alpha = 1/6$} & \multicolumn{2}{c}{$\alpha = 1/3$} \\ 
\cmidrule{2-11}
 & CLT & boot & CLT & boot & CLT & boot & CLT & boot & CLT & boot \\ 
 $ 20 $ & $ 0.1012 $ & $ 0.0466 $ & $ 0.1124 $ & $ 0.0506 $ & $ 0.1052 $ & $ 0.0380 $ & $ 0.1050 $ & $ 0.0466 $ & $ 0.1222 $ & $ 0.0554 $\\
 $ 40 $ & $ 0.0740 $ & $ 0.0542 $ & $ 0.0796 $ & $ 0.0550 $ & $ 0.0772 $ & $ 0.0464 $ & $ 0.0796 $ & $ 0.0552 $ & $ 0.0872 $ & $ 0.0612 $\\
 $ 80 $ & $ 0.0580 $ & $ 0.0526 $ & $ 0.0588 $ & $ 0.0512 $ & $ 0.0652 $ & $ 0.0526 $ & $ 0.0658 $ & $ 0.0552 $ & $ 0.0686 $ & $ 0.0610 $\\
 $ 160 $ & $ 0.0614 $ & $ 0.0596 $ & $ 0.0546 $ & $ 0.0506 $ & $ 0.0584 $ & $ 0.0558 $ & $ 0.0622 $ & $ 0.0572 $ & $ 0.0660 $ & $ 0.0590 $\\
 $ 320 $ & $ 0.0564 $ & $ 0.0558 $ & $ 0.0508 $ & $ 0.0546 $ & $ 0.0502 $ & $ 0.0492 $ & $ 0.0532 $ & $ 0.0536 $ & $ 0.0528 $ & $ 0.0512 $\\
\bottomrule 
\end{tabularx}
\begin{tabularx}{.95\textwidth}{@{\extracolsep{\stretch{1}}}lcccccccccc@{}} 
\multicolumn{11}{l}{Panel C: SV2F} \\
\toprule
$n$  & \multicolumn{2}{c}{$\alpha = -1/3$} & \multicolumn{2}{c}{$\alpha = -1/6$} & \multicolumn{2}{c}{$\alpha = 0$} & \multicolumn{2}{c}{$\alpha = 1/6$} & \multicolumn{2}{c}{$\alpha = 1/3$} \\ 
\cmidrule{2-11}
 & CLT & boot & CLT & boot & CLT & boot & CLT & boot & CLT & boot \\ 
 $ 20 $ & $ 0.0918 $ & $ 0.0400 $ & $ 0.0962 $ & $ 0.0470 $ & $ 0.0986 $ & $ 0.0340 $ & $ 0.1092 $ & $ 0.0516 $ & $ 0.1458 $ & $ 0.0534 $\\
 $ 40 $ & $ 0.0734 $ & $ 0.0548 $ & $ 0.0638 $ & $ 0.0504 $ & $ 0.0748 $ & $ 0.0516 $ & $ 0.0822 $ & $ 0.0670 $ & $ 0.1056 $ & $ 0.0634 $\\
 $ 80 $ & $ 0.0616 $ & $ 0.0584 $ & $ 0.0656 $ & $ 0.0624 $ & $ 0.0686 $ & $ 0.0572 $ & $ 0.0678 $ & $ 0.0620 $ & $ 0.0844 $ & $ 0.0632 $\\
 $ 160 $ & $ 0.0558 $ & $ 0.0544 $ & $ 0.0612 $ & $ 0.0594 $ & $ 0.0572 $ & $ 0.0550 $ & $ 0.0630 $ & $ 0.0606 $ & $ 0.0650 $ & $ 0.0602 $\\
 $ 320 $ & $ 0.0562 $ & $ 0.0562 $ & $ 0.0572 $ & $ 0.0554 $ & $ 0.0512 $ & $ 0.0508 $ & $ 0.0550 $ & $ 0.0574 $ & $ 0.0624 $ & $ 0.0602 $\\
\bottomrule 
\end{tabularx}
\end{center}
{\it Simulation study of the finite sample properties of the test $H_0: \alpha = \alpha_0$ against the alternative $H_1: \alpha \neq \alpha_0$, using the CLT and the local fractional bootstrap. The simulations are done under $H_0$, i.e. we consider the size of the tests. The nominal size is $5\%$ and the numbers shown are the rejection rates of $H_0$ over $5\,000$ Monte Carlo simulations, each with $B= 999$ bootstrap replications. We set $p = 2$ and $\lambda = 1$.} 
\label{tab:size}
\end{table}

\begin{table}
\caption{\it Rejection rates under $H_1$}
\begin{center}
\begin{tabularx}{.95\textwidth}{@{\extracolsep{\stretch{1}}}lcccccccc@{}} 
\multicolumn{9}{l}{Panel A: NoSV} \\
\toprule
$n$  & \multicolumn{2}{c}{$\alpha = -1/3$} & \multicolumn{2}{c}{$\alpha = -1/6$} & \multicolumn{2}{c}{$\alpha = 1/6$} & \multicolumn{2}{c}{$\alpha = 1/3$} \\ 
\cmidrule{2-9}
 & CLT & boot & CLT & boot & CLT & boot & CLT & boot \\ 
 $ 20 $ & $ 0.2308 $ & $ 0.1566 $ & $ 0.1010 $ & $ 0.0708 $  & $ 0.0244 $ & $ 0.0510 $ & $ 0.0284 $ & $ 0.1118 $\\
 $ 40 $ & $ 0.3792 $ & $ 0.2886 $ & $ 0.1646 $ & $ 0.1148 $  & $ 0.0572 $ & $ 0.1030 $ & $ 0.1820 $ & $ 0.3146 $\\
 $ 80 $ & $ 0.6168 $ & $ 0.5428 $ & $ 0.2356 $ & $ 0.1848 $ & $ 0.1314 $ & $ 0.1994 $ & $ 0.5342 $ & $ 0.6212 $\\
 $ 160 $ & $ 0.8688 $ & $ 0.8248 $ & $ 0.3874 $ & $ 0.3228 $ & $ 0.2990 $ & $ 0.3526 $ & $ 0.8660 $ & $ 0.8998 $\\
 $ 320 $ & $ 0.9918 $ & $ 0.9852 $ & $ 0.6160 $ & $ 0.5654 $ & $ 0.5722 $ & $ 0.6074 $ & $ 0.9920 $ & $ 0.9964 $\\
\bottomrule 
\end{tabularx}
\begin{tabularx}{.95\textwidth}{@{\extracolsep{\stretch{1}}}lcccccccc@{}} 
\multicolumn{9}{l}{Panel B: SV1F} \\
\toprule
$n$  & \multicolumn{2}{c}{$\alpha = -1/3$} & \multicolumn{2}{c}{$\alpha = -1/6$} & \multicolumn{2}{c}{$\alpha = 1/6$} & \multicolumn{2}{c}{$\alpha = 1/3$} \\ 
\cmidrule{2-9}
 & CLT & boot & CLT & boot & CLT & boot & CLT & boot \\ 
 $ 20 $ & $ 0.2182 $ & $ 0.1536 $ & $ 0.1162 $ & $ 0.0710 $  & $ 0.0214 $ & $ 0.0472 $ & $ 0.0352 $ & $ 0.1084 $\\
 $ 40 $ & $ 0.3812 $ & $ 0.2914 $ & $ 0.1644 $ & $ 0.1108 $ & $ 0.0530 $ & $ 0.0986 $ & $ 0.1850 $ & $ 0.3052 $\\
 $ 80 $ & $ 0.5834 $ & $ 0.5350 $ & $ 0.2204 $ & $ 0.1958 $ & $ 0.1376 $ & $ 0.1880 $ & $ 0.5300 $ & $ 0.6126 $\\
 $ 160 $ & $ 0.8692 $ & $ 0.8302 $ & $ 0.3652 $ & $ 0.3242 $ & $ 0.2888 $ & $ 0.3458 $ & $ 0.8674 $ & $ 0.9020 $\\
 $ 320 $ & $ 0.9890 $ & $ 0.9874 $ & $ 0.6270 $ & $ 0.5672 $ & $ 0.5598 $ & $ 0.6092 $ & $ 0.9928 $ & $ 0.9948 $\\
\bottomrule 
\end{tabularx}
\begin{tabularx}{.95\textwidth}{@{\extracolsep{\stretch{1}}}lcccccccc@{}} 
\multicolumn{9}{l}{Panel C: SV2F} \\
\toprule
$n$  & \multicolumn{2}{c}{$\alpha = -1/3$} & \multicolumn{2}{c}{$\alpha = -1/6$} & \multicolumn{2}{c}{$\alpha = 1/6$} & \multicolumn{2}{c}{$\alpha = 1/3$} \\ 
\cmidrule{2-9}
 & CLT & boot & CLT & boot & CLT & boot & CLT & boot \\ 
 $ 20 $ & $ 0.1706 $ & $ 0.1166 $ & $ 0.0966 $ & $ 0.0628 $ & $ 0.0258 $ & $ 0.0450 $ & $ 0.0136 $ & $ 0.0686 $\\
 $ 40 $ & $ 0.2834 $ & $ 0.2238 $ & $ 0.1250 $ & $ 0.0988 $ & $ 0.0678 $ & $ 0.0948 $ & $ 0.0928 $ & $ 0.1942 $\\
 $ 80 $ & $ 0.4332 $ & $ 0.3598 $ & $ 0.1814 $ & $ 0.1312 $ & $ 0.1140 $ & $ 0.1446 $ & $ 0.2716 $ & $ 0.3740 $\\
 $ 160 $ & $ 0.6252 $ & $ 0.5890 $ & $ 0.2486 $ & $ 0.1960 $  & $ 0.1998 $ & $ 0.2338 $ & $ 0.5582 $ & $ 0.6262 $\\
 $ 320 $ & $ 0.8676 $ & $ 0.8396 $ & $ 0.3980 $ & $ 0.3468 $  & $ 0.3362 $ & $ 0.4008 $ & $ 0.8320 $ & $ 0.8620 $\\
\bottomrule 
\end{tabularx}
\end{center}
{\it Simulation study of the finite sample properties of the test $H_0: \alpha = 0$ against the alternative $H_1: \alpha \neq 0$, using the CLT and the local fractional bootstrap. The simulations are done under the alternative, i.e. we consider the power of the tests, with the true value of $\alpha$ used in the simulations being the $\alpha$ indicated in the respective column. For the bootstrap, the nominal size is $5\%$, while the CLT has been size-adjusted as explained in the text: the critical value used is the critical value that results in the CLT having the same size as the bootstap test (given in Table \ref{tab:size}). The numbers shown are the rejection rates of $H_0$ over $5\,000$ Monte Carlo simulations, each with $B= 999$ bootstrap replications. We set $p = 2$ and $\lambda = 1$.} 
\label{tab:power}
\end{table}

\section{Empirical applications}
\label{sec:empirical}
In this section, we apply the local fractional bootstrap method \mikko{presented} above to two relevant empirical \mikko{data sets}. As we saw in the previous section, the bootstrap method is crucial for \mikko{achieving the correct empirical size of} the hypothesis test $H_0: \alpha = \alpha_0$, especially when the number of observations \mikko{$n$} is small. In both \mikko{of} our applications\mikko{,} \mikko{theoretical arguments suggest specific null hypotheses} to be true, and we examine how the CLT and bootstrap \mikko{fare} in confirming or rejecting these hypotheses.

\subsection{\mikko{High-frequency futures price data}}
Here, we consider testing $H_0: \alpha  = 0$ for the price of a financial asset, the E-mini \mikko{S\&P 500} futures contract. Note that\mikko{, at least theoretically,} on no-arbitrage grounds \citep[Theorem 7.2]{DS94} one would expect $H_0$ to be true, since $\alpha \neq 0$ implies that the $\BSS$ process is \mikko{not a semimartingale} \citep[e.g.][]{CHPP13,bennedsen16a}.

Our data consists of \mikko{high-frequency}\footnote{\mikko{The data has been recorded with one-second time stamp precision.}} observations of the E-mini \mikko{S\&P 500} futures contract\mikko{, traded on CME Globex electronic trading platform,} from January 2, 2013 until December 31, 2014\mikko{,} excluding weekends and holidays. This results in 516 trading days, 18 of which were not full trading days; we removed these to arrive at a total of 498 days in our sample. Although the E-mini \mikko{S\&P 500 contract} is traded almost around the clock\mikko{,} we restrict attention to the most liquid time period which is when the NYSE is open, i.e. to the $6.5$ hours from $9.30$ a.m.\ to $4$ p.m.\ EST. We \mikko{re}sample the price at different frequencies, $\Delta \in \{2, 5, 10, 15\}$ minutes, which results in $n = 196, 79, 40, 27$ daily prices, respectively. We then test $H_0: \alpha = 0$ for each day using both the CLT and the local fractional bootstrap against the \mikko{two-sided} alternative $H_1: \alpha \neq 0$. The \mikko{results,} averaged over the $498$ days in our sample, are presented in Figure \ref{fig:ES}. We see that when we sample the price \mikko{frequently} ($\Delta = 2$ minutes)\mikko{,} both methods \mikko{reject} often ($16.5\%$ and $14.1\%$, respectively) indicating that \mikko{the number of days when $H_0:\alpha = 0$ is rejected is significant}. This result \mikko{may seem surprising}, but is likely due to market microstructure (MMS) noise effects: \mikko{at very short time scales, high-frequency data often exhibits negative autocorrelations that are compatible with the alternative hypothesis $\alpha<0$.}

When sampling at lower frequencies, i.e. \mikko{$\Delta$ being at least 5 minutes}, one expects the MMS effects to be negligible and \mikko{rejections of $H_0$ should occur at the nominal rate}. Figure \ref{fig:ES} indeed shows that we reject $H_0$ less often in these cases; we also observe that the bootstrap method rejects less often than \mikko{the} CLT. Indeed, the bootstrap is closer to the nominal $5\%$ rejection rate which we would expect under \mikko{$H_0$}. In the \mikko{case} $\Delta = 15$ minutes\mikko{,} we \mikko{have only} $n = 27$ observations \mikko{per} day\mikkel{; considering the Monte Carlo evidence above, we expect the CLT to be oversized in this case. As seen in the figure, the CLT indeed rejects $H_0$ more often here ($11.4\%$) while the bootstrap essentially retains the nominal size, rejecting on $4.6\%$ of the days}. This is encouraging as any MMS effects \mikko{should be} negligible \mikko{at this time scale}.

\begin{figure}[!t] 
\centering 
\includegraphics[scale=0.83]{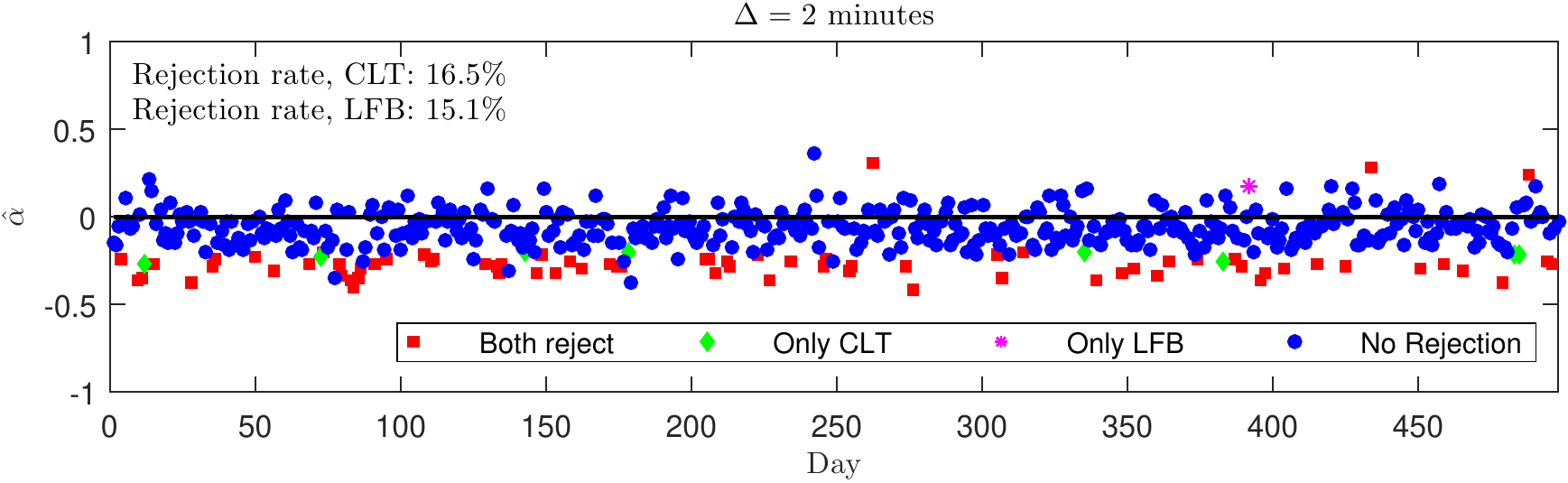} \\
\includegraphics[scale=0.83]{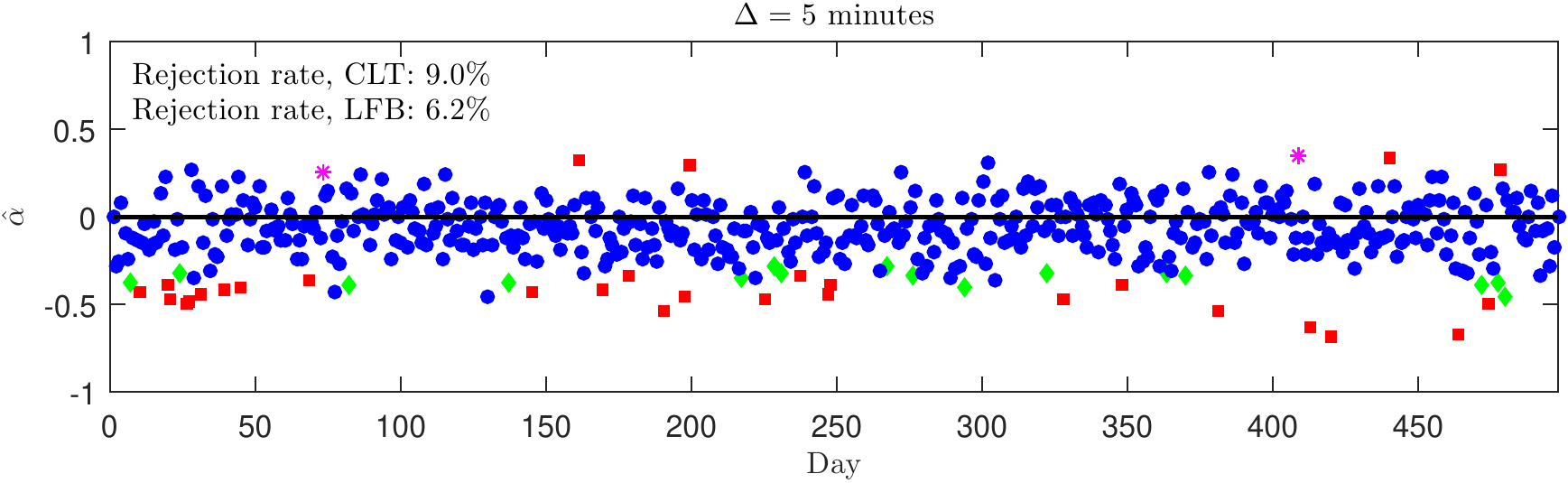} \\
\includegraphics[scale=0.83]{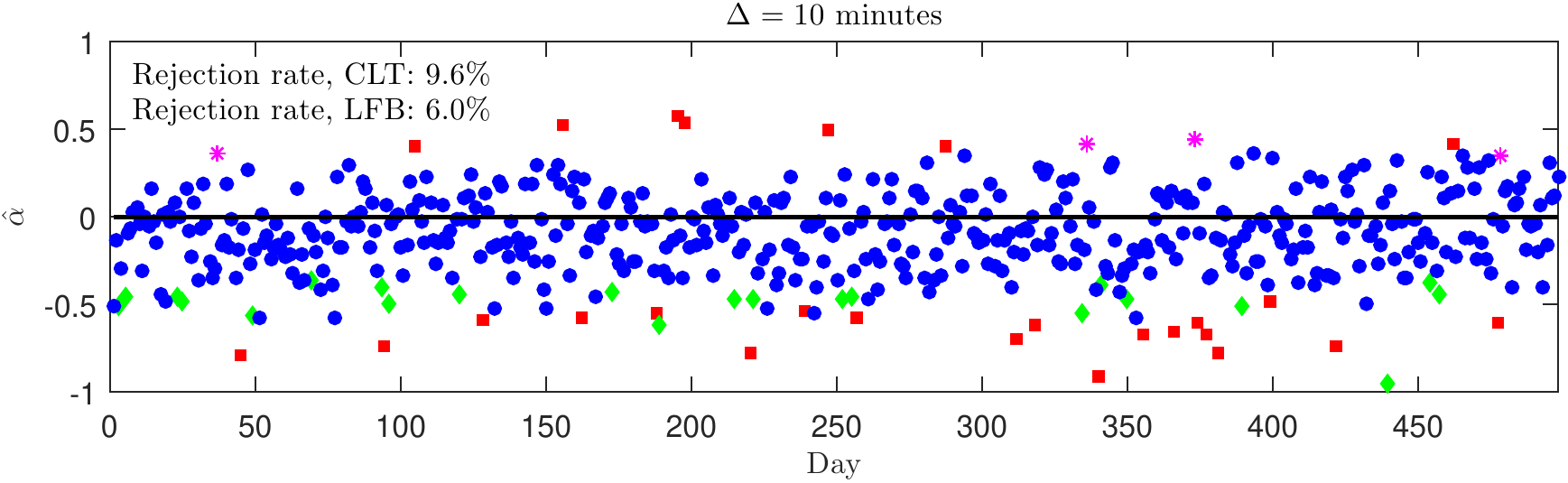} \\
\includegraphics[scale=0.83]{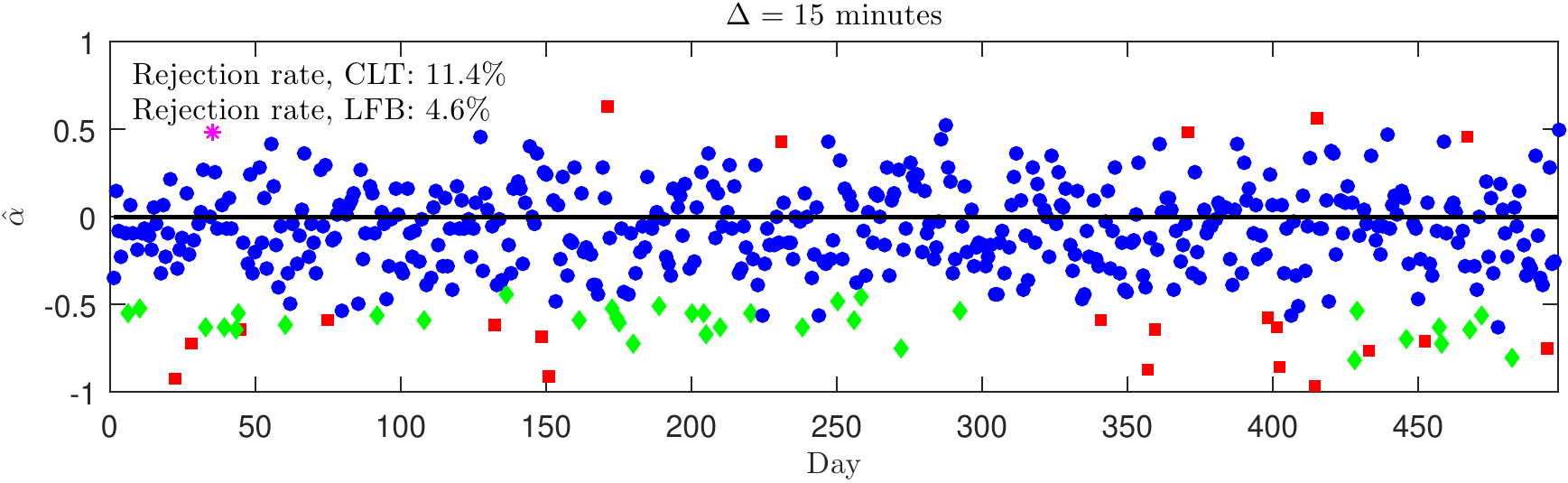}
\caption{\mikkel{Estimates of $\alpha$  and test of $H_0: \alpha = 0$ against $H_1: \alpha \neq 0$ from $498$  intraday time series of log-prices from the E-mini data set, sampled every $\Delta$ period. The plots depict the estimated value of $\alpha$ from a given day. The particular days where $H_0$ was rejected by both the CLT and LFB are shown by red squares; the days where only the CLT, but not the LFB, rejected are green diamonds; the days where only the LFB, but not the CLT, rejected are magenta asterisks; and the days where no method rejects $H_0$ are blue circles. We used $B = 999$ bootstrap replications and the black horizontal line indicates the null value $\alpha = 0$.}}
\label{fig:ES}
\end{figure}

\subsection{\mikko{Turbulence data}}

\mikko{In our second application, we study a time series of one-dimensional hot-wire anemometer measurements of the longitudinal component of a turbulent velocity field in the atmospheric boundary layer, measured $35$ meters above ground level.} The \mikko{time series} consists of $2 \times 10^7$ \mikko{observations}, sampled at a \mikko{rate} of $5$ kHz. In other words, \mikko{there are} $20$ million \mikko{observations}, measured over a period of $T = 4\,000$ seconds with $5\,000$ observations recorded \mikko{per} second. This time series was also studied in \cite{CHPP13} and \cite{BNPS14}, and we refer to \cite{dhruva00} for further details on how it was recorded.

\mikko{When timewise data on turbulence is modeled using a $\BSS$ process whose kernel function satisfies Assumption 1(i), Kolmogorov's $5/3$ scaling law \citep{kolmogorov41} for fully developed turbulence, assuming Taylor's frozen field hypothesis \citep{Taylor1938}, is compatible with the parameter value $\alpha = -1/6$ at intermediate time scales that correspond to the so-called \emph{inertial range}; see also \citet{MS16}. By analyzing the spectral density of the time series, for this data the inertial range was found to be approximately between $0.1$ Hz and $200$ Hz \cite[Section 5]{CHPP13}.}

\mikko{As in the previous application, we experiment with resampling the data at various frequencies $f$, varying $f$ to study the roughness properties of the time series at different time scales. More specifically, we vary $f$ between $1$ Hz and $200$ Hz to include time scales both firmly within and on the border of the inertial range. The time increment $\Delta$ used in resampling is related to $f$ by $\Delta = 1/f$. Motivated by Kolmogorov's scaling law, we formulate} $H_0: \alpha = -1/6$ \mikko{and test it against the two-sided alternative $H_1: \alpha \neq -1/6$. In our analysis, we divide the sample period (of $4\,000$ seconds) into $M=400$ sub-periods of 10 seconds. We conduct the test on each sub-period individually, treating them as separate measurements of the same phenomenon, which seems reasonable given the putative stationarity of the time series. Note that after resampling at frequency $f$, the number of observations covering each sub-period is $n=10f$.}

\mikko{Figure \ref{fig:brook} presents results on} the rejection rate of $H_0$, which is \mikko{the relative frequency} of rejections over the $M = 400$ \mikko{sub-periods}. As expected\mikko{, $H_0$ is often rejected} when \mikkel{the sampling rate is on the border of the inertial range (cf. the results for $f = 200$ Hz). When firmly inside the inertial range ($f = 20$ Hz), the null hypothesis is rejected in roughly $4\%$ of the sub-periods for both methods. At these sampling frequencies, the CLT- and bootstrap-based tests} largely agree; this is as expected since there are plenty of observations. 

\mikko{The results change} as the sampling frequency \mikkel{is lowered, resulting in fewer observations. Indeed, we see that the CLT-based test yields rejection rates of $10.8\%$ and $17.3\%$ at sampling frequencies $1$ and $2$ Hz, respectively}, while the \mikko{bootstrap-based test} rejects \mikko{roughly} at the nominal $5\%$ \mikko{rate,} as we would expect from a correctly\mikko{-}sized test when $H_0$ is true. \mikko{As seen also in the previous application, at low sampling frequencies (here $1$ Hz, which leads to $n = 10$ observations),} the COF estimator seems to be severely biased (\mikko{the} mean \mikko{of the} estimates of $\alpha$ \mikko{is} around $-0.475$). In this case\mikko{,} the CLT\mikko{-based test} starts rejecting $H_0$ at a\mikko{n unplausibly} high rate (around $17\%$) while the \mikko{bootstrap-based test is more} conservative \mikkel{with rejection rate around $6\%$}. As we \mikko{would} still expect \mikko{the null hypothesis} $\mikko{H_0:\alpha = -1/6}$ to \mikko{be actually true at} this sampling frequency, it is \mikko{reassuring} that the bootstrap\mikko{-based test} is so close to the nominal \mikko{rate} $5\%$ in this, \mikko{arguably} extreme\mikko{,} case. \mikko{However, this comes with the caveat that $n = 10$ observations might simply be too few to draw any definite conclusion on.}

\begin{figure}[!t] 
\centering 
\includegraphics[scale=0.81]{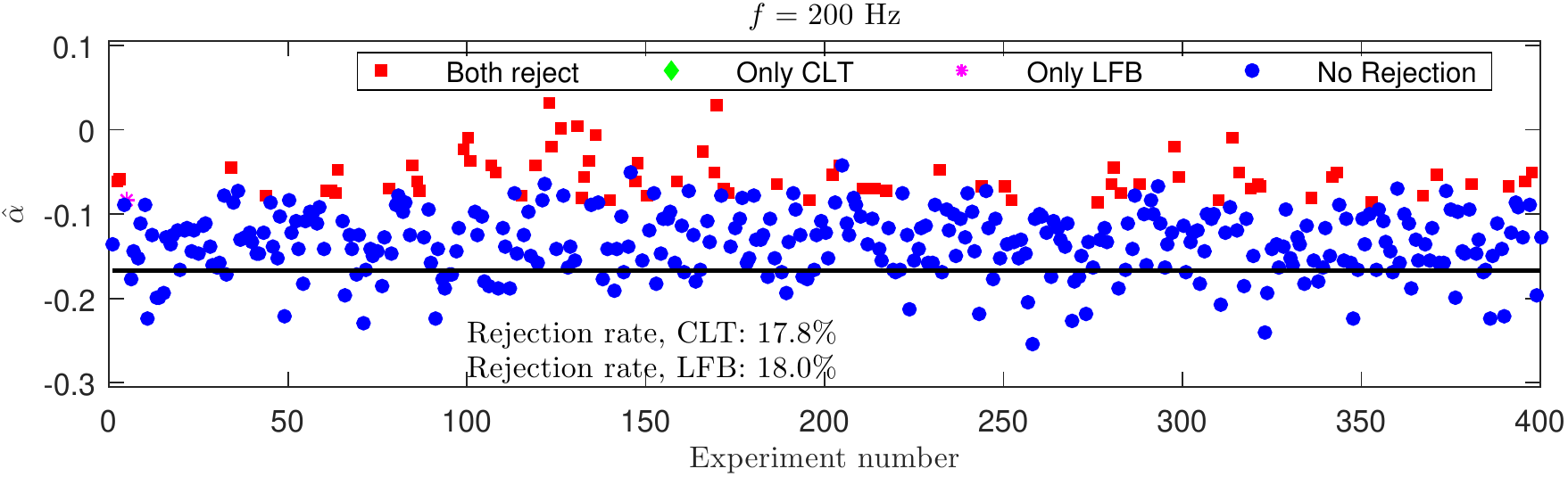} \\
\includegraphics[scale=0.81]{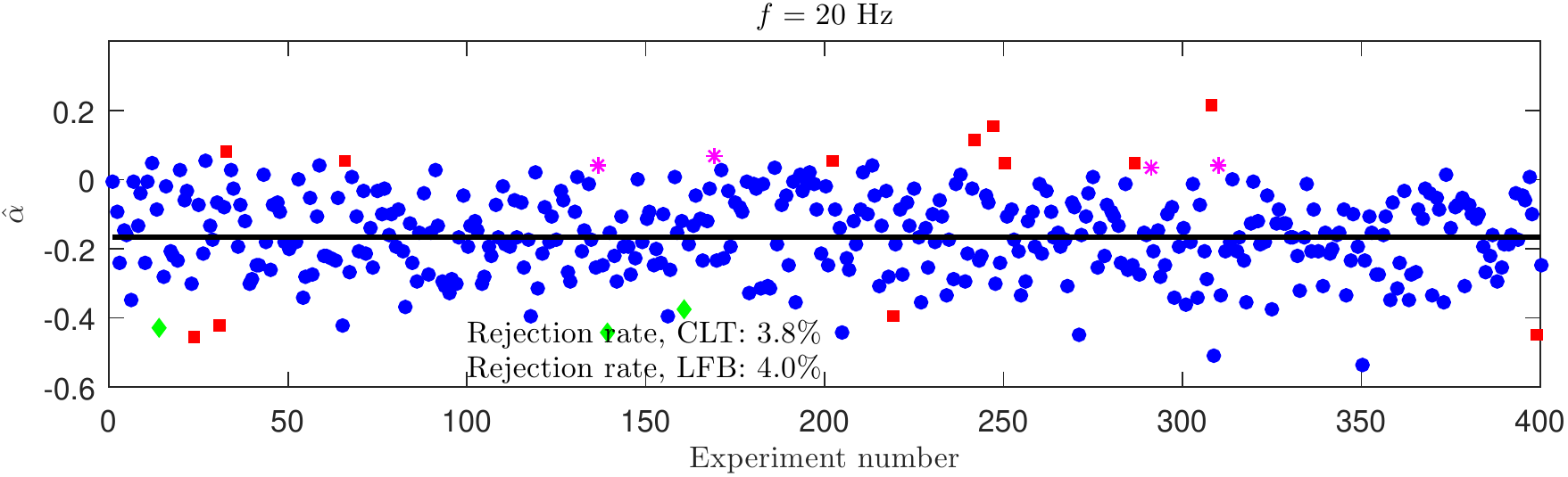} \\
\includegraphics[scale=0.81]{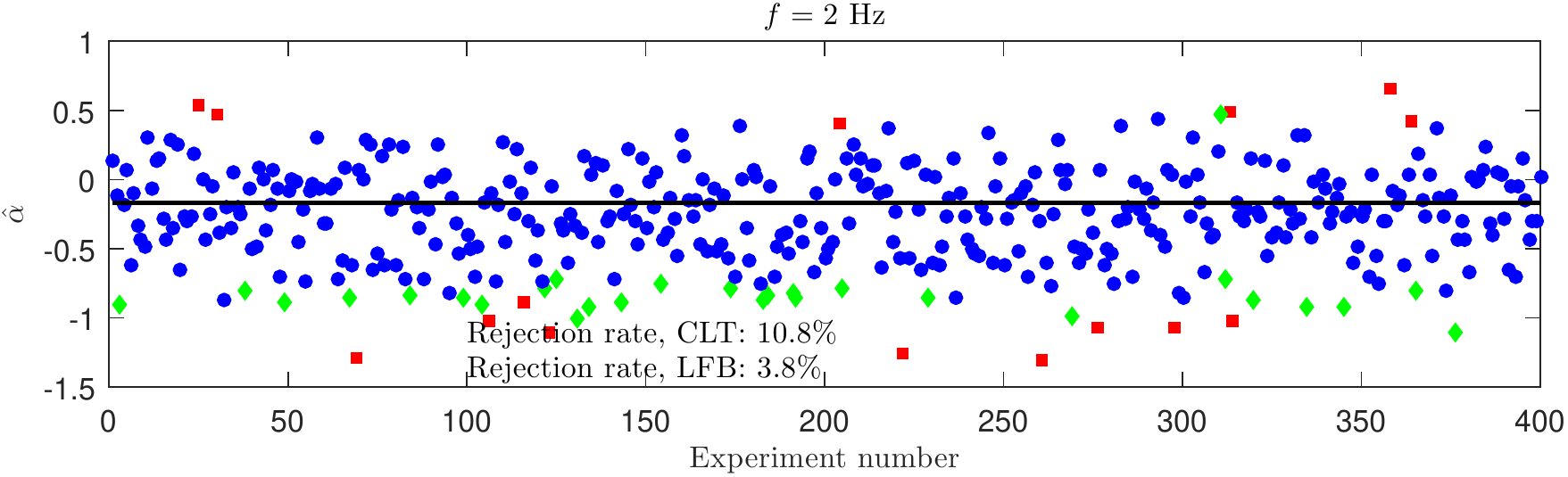} \\
\includegraphics[scale=0.81]{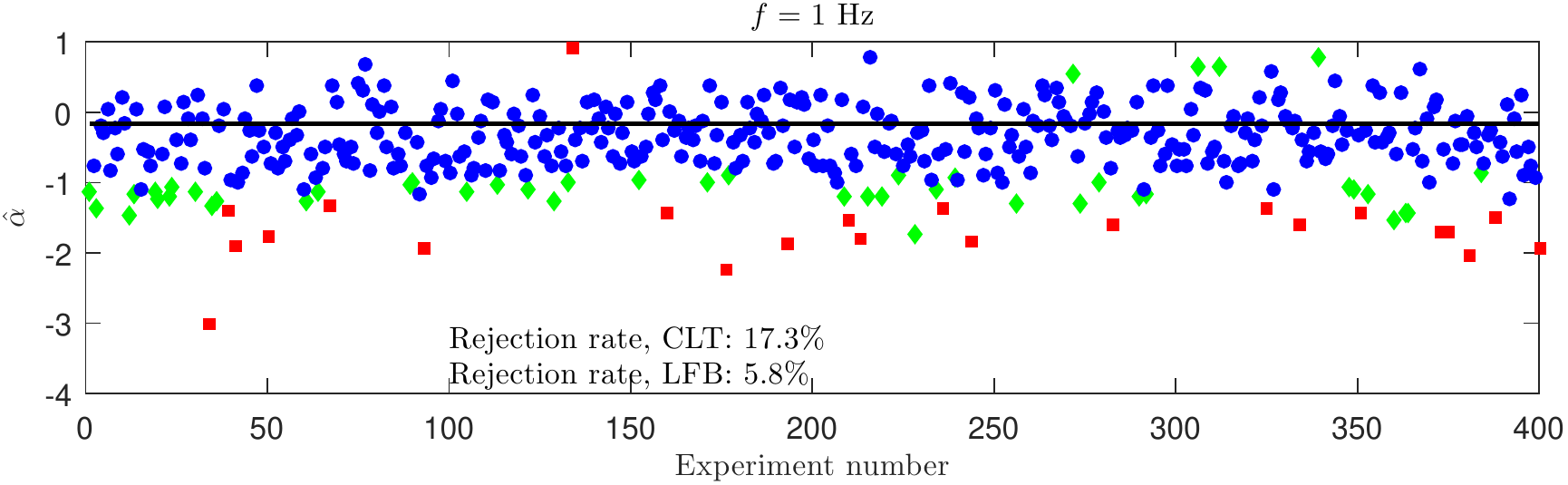}
\caption{\mikkel{Estimates of $\alpha$  and test of $H_0: \alpha = -1/6$ against $H_1: \alpha \neq -1/6$ from $400$  experiments using the turbulence data described in the text, see also \cite{dhruva00}. The data is sampled at frequency $f$. The plots depict the estimated value of $\alpha$ from a given experiment. The particular experiments where $H_0$ was rejected by both the CLT and LFB are shown by red squares; the experiments where only the CLT, but not the LFB, rejected are green diamonds; the experiments where only the LFB, but not the CLT, rejected are magenta asterisks; and the experiements where no method rejects $H_0$ are blue circles. We used $B = 999$ bootstrap replications and the black horizontal line indicates the null value $\alpha = -1/6$}.}
\label{fig:brook}
\end{figure}

\section{Conclusion}

\label{sec:concl} We have proposed a novel bootstrap method \mikko{of conducting} inference on the roughness index \mikko{$\alpha$} of a Brownian semistationary
process using the change-of-frequency estimator. \mikko{While our simulation study indicates that the performance of both the CLT- and bootstrap-based tests} is generally good, the bootstrap approach \mikko{improves the size properties of the test of $H_0 : \alpha = \alpha_0$ when the number of observations is moderate or small.}

As an application, we applied the method to test \mikko{for} $H_0: \alpha = 0$ \mikko{with a} time series of \mikko{intraday prices} of the E-mini \mikko{S\&P 500} futures contract and to test \mikko{for} $H_0: \alpha = -1/6$ with a time series of measurements of atmospheric turbulence. With both data sets, we observed what the simulation results \mikko{already} indicated: the CLT rejects the \mikko{respective} null hypotheses, that we expect to be true on theoretical grounds, too often when \mikko{the number of observations is limited}, while the local fractional bootstrap retains the correct size. We conclude that the local fractional bootstrap is a powerful alternative to the CLT when \mikko{drawing} inference on the roughness index $\alpha$, \mikko{and it appears to be essential at lower observation frequencies.} 

Finally, we note that \mikko{while} in this paper we \mikko{have} focused on $\BSS$ processes, \mikko{the local fractional bootstrap method should be applicable to other ``fractional'' processes such as} the fractional Brownian motion (fBm), fractional Ornstein-Uhlenbeck process, \mikko{and the like. We leave such extensions for future work.}

\newpage

{\small 
\bibliographystyle{chicago}
\bibliography{bootBSSbib-v04}
}

\appendix

\section{Simulation design}\label{sec:sim}

In our Monte Carlo study \mikko{presented} above we \mikko{have simulated $n+1\in \mathbb{N}$ equidistant observations $X_0,X_{1/n},X_{2/n},\ldots,X_1$} of
the $\BSS$ process
\begin{equation}
X_{t}=\int_{-\infty }^{t}g(t-s)\sigma_sdW_{s}  \label{eq:Xbss}
\end{equation}%
on the time interval $[0,1]$\mikko{.} Recall that we take $g(x)=x^{\alpha }e^{-\lambda x}$ where $%
\lambda >0$ and $\alpha \in \left( -\frac{1}{2},\frac{1}{2}\right) $.
Simulation of $X$ is \mikko{not straightforward} as the process \mikko{is typically neither Gaussian nor Markovian, which rules out both exact and recursive simulation schemes.} \mikko{However, as shown in \cite{BLP15}, the process $X$ can be simulated efficiently and accurately using the so-called \emph{hybrid scheme}, which is based on approximating $X_t$ by a Riemann sum plus Wiener integrals of a power function that mimicks the steep behavior of $g$ at zero.
In particular, 
the hybrid scheme improves significantly simulation accuracy compared to any approximation using
merely Riemann sums.}

\mikko{To simulate the observations $X_0,X_{1/n},X_{2/n},\ldots,X_1$, the hybrid scheme approximates $X_{i/n}$, $%
i=0,1,\ldots,n$,} by
\begin{align*}
\mikko{X^n_{i/n} := \check{X}^n_{i/n} + \hat{X}^n_{i/n},}
\end{align*}
where
\begin{align}
\mikko{\check{X}^n_t} & := \mikko{\sum_{k=1}^\kappa L_g\bigg( \frac{k}{n} \bigg) \sigma_{t-k/n} \int_{t-\frac{k}{n}}^{t-\frac{k}{n}+\frac{1}{n}}(t-s)^\alpha dW_s,} \label{eq:checkX} \\
\mikko{\hat{X}^n_t} & := \mikko{\sum_{k=\kappa+1}^{N_n} g\bigg(\frac{b_k^*}{n}\bigg)\sigma_{t- k/n} ( W_{t-k/n+1/n} - W_{
t-k/n}). } \label{eq:hatX}
\end{align}
The number $N_n := \lfloor n^{1+\delta} \rfloor$, for some $\delta >0$,
determines the truncation "towards minus infinity", while $\kappa \geq 0$
denotes the number of terms that are simulated directly \mikko{via} Wiener
integrals, cf\mikko{.} \eqref{eq:checkX}. As shown in \cite{BLP15}, $\kappa = 1$ suffices when $%
\alpha < 0$, but we need $\kappa = 3$ when $\alpha$ is close to $\frac{1}{2}$%
. In the simulations, we therefore choose $\kappa = 1$ when $\alpha < 0$ and
$\kappa =3$ when $\alpha > 0$. We also let $\delta = 0.5$. The numbers
\begin{align*}
b_k^* = \left( \frac{k^{\alpha+1} - (k-1)^{\alpha+1}}{\alpha+1}%
\right)^{1/\alpha},
\end{align*}
$k=1\ldots, N_n$, are \mikko{the} optimal points\footnote{%
\mikko{In the sense of asymptotic mean-square error.}} \mikko{to evaluate $%
g $ at}; see Proposition 2.2 of \cite{BLP15}. We refer to \cite{BLP15} for implementation of the algorithm used to simulate %
\eqref{eq:checkX} and \eqref{eq:hatX} exactly while simulateneously
simulating \mikko{$\sigma_{i/n - k/n}$}, $i,k = 0, 1,
\ldots,$ which may be correlated with $W.$

For the stochastic volatility process $\sigma = \mikko{(}\sigma_t\mikko{)}_{t \in \mathbb{R%
}}$, we consider three different specifications: (i) constant volatility,
labeled NoSV; (ii) one-factor stochastic volatility, labeled SV1F; and (iii)
two-factor stochastic volatility, labeled SV2F. For the NoSV model we take
for $t\in \mathbb{R}$,
\begin{eqnarray*}
\sigma_t &=& 1,
\end{eqnarray*}
while we in the SV1F model take, following \cite{BNHLS08},
\begin{eqnarray*}
\sigma_t &=& \exp(\beta_0 + \beta_1 \tau_t), \\
d\tau_t &=& \xi \tau_t dt + dB_t, \\
\mikko{d[W,B]_t} &=& \rho dt,
\end{eqnarray*}
where $B$ is a standard Brownian motion and $\beta_1 =0.125,$ $\xi = -0.025,$
$\beta_0 = \frac{\beta_1^2}{2\xi} = -0.3125$ and $\rho = -0.3.$ Lastly, for
the SV2F model we take, following \cite{HT05} and
 \cite{BNHLS08},
\begin{eqnarray*}
\sigma_t &=& \mathnormal{s}\textnormal{-}\exp(\beta_0 + \beta_1\tau_{1t} + \beta_2
\tau_{2t}), \\
d\tau_{1t} &=& \xi_1 \tau_{1t} dt + d\mikko{B^1_{t}}, \\
d\tau_{2t} &=& \xi_2 \tau_{2t} dt + (1+\phi\tau_{2t})d\mikko{B^2_{t}}, \\
\mikko{d[W,B^1]_t} &=& \rho_1dt, \\
\mikko{d[W,B^2]_t} &=& \rho_2dt,
\end{eqnarray*}
where \mikko{$B^1$}, \mikko{$B^2$} are standard Brownian motions and the function $%
\mathnormal{s}\textnormal{-}\exp$ is given by
\begin{equation*}
\mathnormal{s}\textnormal{-}\exp (x) =\left\{
\begin{array}{ll}
\exp(x), & \quad x \leq \log(1.5), \\
\frac{3}{2}\sqrt{1-\log(1.5) + x^2/\log(1.5)}, & \quad x > \log(1.5)%
\end{array}%
\right.
\end{equation*}%
and the parameters are set to $(\beta_0, \beta_1, \beta_2)^T
=(-1.20,0.040,1.50)^T,$ $(\xi_1, \xi_2)^T = (-0.00137,-1.386)^T,$ $\phi =
0.250$ and $\rho_1 = \rho_2 = -0.30.$

We \mikko{note that} in the NoSV case the process $X$ is Gaussian and can thus be
simulated exactly using a Cholesky decomposition of its variance-covariance
matrix, which is what we do in our simulations. The stochastic processes of
SV1F and SV2F can be simulated exactly using methods in \cite{glasserman03},
see also the Simulation Appendix to \cite{BNHLS08}.

\section{Expressions for $\Lambda_p$, $\lambda_{p,n}^{i,j}$, and $\mu(p,\upsilon)_t^n$}\label{app:derive}
In both the \mikko{ordinary} CLT \eqref{CLT-Alpha-Test} as well as the bootstrap CLT there are various terms that are necessary to derive when implementing the methods. In particular, the \mikko{ordinary} CLT requires calculation of $\Lambda_p = \{ \lambda_p^{i,j} \}_{1\leq i,j\leq 2}$ while the bootstrap CLT requires calculation of $\{\lambda_{p,n}^{i,j} \}_{1\leq i,j\leq 2}$ as well as $\mu(p,\upsilon)_t^n =\mathbb{E}^{\ast }\left( V\left( B^{H};p,\upsilon
\right)^n_{t}\right)$ for $\upsilon = 1,2$, see Remark \ref{rem:bootFeasible}. Although the calculations involved are \mikko{quite straightforward,} they are tedious. For \mikko{the convenience of the reader,} we supply the expressions for these terms here. In what follows we denote by $n = \lfloor t/\Delta_n \rfloor$ the total number of observations.

Recall the specifications
\begin{eqnarray*}
\lambda _{p,n}^{11} &=&\Delta
_{n}^{-1}var\left( \bar{V}\left( B^{H};p,1\right)^n_{1}\right), \\
\lambda _{p,n}^{22}&=&\Delta
_{n}^{-1}var\left( \bar{V}\left( B^{H};p,2\right)^n_{1}\right) , \\
\lambda _{p,n}^{12} &=&\Delta
_{n}^{-1}cov\left( \bar{V}\left( B^{H};p,1\right)^n_{1},\bar{V}%
\left( B^{H};p,2\right)^n _{1}\right) ,
\end{eqnarray*}%
and $\lambda_p^{i,j} = \lim_{n\rightarrow \infty} \lambda _{p,n}^{i,j}$. Analytical expressions are only known for $p=2$, which is arguably the most relevant in empirical applications as it corresponds to using squared increments when calculating power variations. We therefore only consider $p=2$ here.

Let $\rho^H_{\upsilon_1,\upsilon_2}$ be the correlation function between the second order increment\mikko{s} of \mikko{the} fractional Brownian motion with Hurst index $H$, calculated at lag $\upsilon_1$, and the second order increment calculated at lag $\upsilon_2$. In other words
\begin{eqnarray*}
\rho^H_{\upsilon_1,\upsilon_2}(h) &:=& Corr \left(B^H_{i+h} - 2B^H_{i+h-\upsilon_1} + B^H_{i+h-2\upsilon_1}, B^H_{i} - 2B^H_{i-\upsilon_2} + B^H_{i-2\upsilon_2}\right), \quad h \in \mathbb{Z}.
\end{eqnarray*}
We will need the combinations $(\upsilon_1,\upsilon_2) = (1,1), (2,2), (1,2)$ and we give them for reference:
\begin{eqnarray*}
\rho^H_{1,1}(h) &:=& \frac{1}{2(4-2^{2H})}\left( -|h-2|^{2H} + 4 |h-1|^{2H} - 6|h|^{2H} + 4|h+1|^{2H} - |h+2|^{2H} \right), \\
\rho^H_{2,2}(h) &:=& \frac{1}{2\left(4\cdot 2^{2H}-4^{2H}\right)}\left( -|h-4|^{2H} + 4 |h-2|^{2H} - 6|h|^{2H} + 4|h+2|^{2H} - |h+4|^{2H} \right), \\
\rho^H_{1,2}(h) &:=& \frac{-|h-2|^{2H} +2|h-1|^{2H}+ |h|^{2H} - 4|h+1|^{2H} + |h+2|^{2H} +2|h+3|^{2H} - |h+4|^{2H} }{2\sqrt{4-2^{2H}}\sqrt{4\cdot 2^{2H}-4^{2H}}}.
\end{eqnarray*}

Brute force calculations will yield
\begin{eqnarray*}
\lambda _{2,n}^{11} &=& 2 \Delta_n^{4H} \sum_{i=2}^n \sum_{j=2}^n \rho^H_{1,1}(i-j)^2 , \\
\lambda _{2,n}^{22} &=& 2\Delta_n^{4H}\sum_{i=4}^n \sum_{j=4}^n  \rho^H_{2,2}(i-j)^2, \\
\lambda _{2,n}^{12}&=& \lambda _{2,n}^{21} =  2\Delta_n^{4H} \sum_{i=2}^n \sum_{j=4}^n\rho^H_{1,2}(i-j)^2.
\end{eqnarray*}%

In the feasible implementation in Remark \ref{rem:bootFeasible} we will actually need the unnormalized variants of $\lambda_ {2,n}^{ij}$ which are
\begin{eqnarray*}
Var\left( V\left( B^{H};p,1\right)^n_{1}\right) &=& \Delta_{n} \lambda _{2,n}^{11} \cdot (4-2^{2H})^2  , \\
Var\left( V\left( B^{H};p,2\right)^n_{1}\right) &=& \Delta_{n} \lambda _{2,n}^{22} \cdot (4\cdot2^{2H} - 4^{2H})^2, \\
Cov\left( V\left( B^{H};p,1\right)^n_{1},V\left( B^{H};p,2\right)^n _{1}\right) &=&  \Delta_{n} \lambda _{2,n}^{12} \cdot (4-2^{2H})(4\cdot2^{2H} - 4^{2H}).
\end{eqnarray*}%

To arrive at expressions for $\lambda_p^{ij}$ we can either take limits in the above or use the theory in \cite{NPP11} \citep[see e.g.][]{BNCP12} to get
\begin{eqnarray*}
\lambda _{2}^{11} &=& 2 + 4\sum_{h=1}^{\infty} \rho^H_{1,1}(h)^2, \\
\lambda _{2}^{22}&=&2 + 2^{-4H+2} \sum_{h=1}^{\infty} \big[ \rho^H_{1,1}(h-2) + 4\rho^H_{1,1}(h-1) + 6\rho^H_{1,1}(h) + 4\rho^H_{1,1}(h+1) + \rho^H_{1,1}(h+2) \big]^2, \\
\lambda _{2}^{12} &=& \lambda _{2}^{21}  = 2^{3-2H}\big(\rho^H_{1,1}(1) +1\big)^2 + 2^{2-2H}\sum_{h=0}^{\infty} \big[ \rho^H_{1,1}(h) + 2\rho^H_{1,1}(h+1) + \rho^H_{1,1}(h+2) \big]^2.
\end{eqnarray*}%

\mikko{Finally}, we need to calculate $\mu(2,1)_t^n$ and $\mu(2,2)_t^n$. \mikko{Straightforward} calculations yield
\begin{eqnarray*}
\mu(2,1)_t^n &=& (n-1) \Delta_n^{2H}\big(4-2^{2H}\big), \\
\mu(2,2)_t^n &=& (n-3)\Delta_n^{2H}\big(4\cdot 2^{2H} - 4^{2H}\big).
\end{eqnarray*}%

\section{Proofs}\label{sec:proofs}

\begin{proof}[Proof of Proposition \ref{prop:a0}]
The proofs of the statements follow the ones referenced in \cite{CHPP13} almost verbatim, which in turn relies on results from \cite{BNCP12}. We note that in the case of $\alpha = 0$ the increments of the $\BSS$ process are asymptotically uncorrelated. Indeed, we have by Assumption 2'(i) \citep[cf. equation (2.6) in][]{CHPP13}
\begin{eqnarray*}
r_{n}(j) :=  \textnormal{corr}\big(G_{(j+1)\Delta_n} - 2G_{j\Delta_n} + G_{(j-1)\Delta_n}, G_{\Delta_n} - 2G_0 + G_{-\Delta_n}\big)   \stackrel{n\rightarrow \infty}{\longrightarrow} \rho_2(j), \quad j\geq 0,
\end{eqnarray*}
where $\rho_2(j)$  is the correlation function of the second order increments of a Brownian motion. Therefore, 
\begin{eqnarray*}
\rho_2(j) = 0, \quad j \geq 2.
\end{eqnarray*}
This uncorrelatedness simplifies matters, for instance when we need to calculate $\Lambda_p$  \citep[see pages 85-86 in][]{BNCP12}. The proof of Proposition \ref{prop:a0} has two main parts, see the similar proofs in \cite{BNCP11,BNCP12}. First, we need to show the existence of a sequence $r(j)$, such that
\begin{eqnarray}\label{eq:rseq}
|r_{n}(j)| \leq Cr(j), \quad \sum_{j=1}^{\infty} r(j)^2 < \infty, 
\end{eqnarray}
where $C>0$, see page 75 in \cite{BNCP12} for the case of $\alpha \neq 0$. Given that  \citep[e.g.][]{BNCP12}
\begin{eqnarray*}
r_{n}(j) = \frac{- R\left(\frac{j+2}{n}\right) + 4R\left(\frac{j+1}{n}\right) - 6R\left(\frac{j}{n}\right) + 4R\left(\frac{j-1}{n}\right) - R\left(\frac{j-2}{n}\right) }{4R\left(\frac{1}{n}\right)- R\left(\frac{2}{n}\right)},
\end{eqnarray*}
it is not difficult to show, using Assumptions 2'(i)-(iii) and the approach from the proof of Lemma 1 in  \cite{BNCP09}, that the sequence
\begin{eqnarray*}
r(j) = (j-1)^{-\beta}, \quad j \geq 2,
\end{eqnarray*}
will suffice as the sequence \mikko{in} \eqref{eq:rseq}, where $\beta$ is the parameter from Assumption 2'(ii). We now turn to the second part of the proof. Define
\begin{align*}
\pi^n(A) := \frac{\int_A \left(g(x+2\Delta_n) - 2g(x+\Delta_n) + g(x)\right)^2dx}{\int_0^{\infty} \left(g(x+2\Delta_n) - 2g(x+\Delta_n) + g(x)\right)^2dx}, \quad A \in \mathcal{B}\left(\mathbb{R}\right).
\end{align*}
The \mikko{only remaining} thing to show \mikko{is} to ensure that the limit theorems apply for stochastic $\sigma$, is that for all $\epsilon >0$ we have $\pi^n((\epsilon,\infty)) \rightarrow 0$ as $n \rightarrow \infty$. Using Assumption 1' (ii) and arguments as the ones in  \cite{BNCP12} page 74, we easily deduce this property. This concludes the proof.
\end{proof}

\begin{proof}[Proof of Lemma \ref{LemmaBoot-New-moments}]
(i) Given (\ref{one}), (\ref{Boot-New-Power-Var}), and (\ref%
{Boot-New-Power-Var-Bar}) the result follows. In particular, we have that
\begin{eqnarray*}
\mathbb{E}^{\ast }\left( \bar{V}^{\ast }\left( X,B^{H};p,\upsilon \right)^n _{t}\right) &=&\frac{\left\vert \widehat{\sigma }(p,\upsilon)_t^n\right\vert ^{p}}{\overline{\mu }(p,\upsilon )_t^n}%
\mathbb{E}^{\ast }\left( \bar{V}\left( B^{H};p,\upsilon \right)^n_{t}\right) \\
&=&\left\vert \widehat{\sigma }(p,\upsilon)_t^n\right\vert ^{p}
\end{eqnarray*}

(ii) Given (\ref{one}), (\ref{Boot-New-Power-Var}), and (\ref%
{Boot-New-Power-Var-Bar}), we can write
\begin{eqnarray*}
&&Var^{\ast }\left( \Delta _{n}^{-1/2}\bar{V}^{\ast }\left(
X,B^{H};p,1\right)^n_{t}\right) \\
&=&\Delta _{n}^{-1}\left( \frac{\left\vert \widehat{\sigma }(p,1)_t^n\right\vert ^{p}}{\overline{\mu }(p,1)_t^n}\right)
^{2}Var^{\ast }\left( \bar{V}\left( B^{H};p,1\right)^n_{t}\right) .
\end{eqnarray*}

(iii) Follows similarly as the proof of Lemma \ref{LemmaBoot-New-moments} part
(ii). (iv) Given (\ref{Boot-New-Power-Var-Bar}), we can write
\begin{eqnarray*}
&&Cov^{\ast }\left( \Delta _{n}^{-1/2}\bar{V}^{\ast }\left(
X,B^{H};p,1\right)^n_{t},\Delta _{n}^{-1/2}\bar{V}^{\ast
}\left( X,B^{H};p,2\right)^n_{t}\right) \\
&=&\Delta _{n}^{-1}\left( \frac{\left\vert \widehat{\sigma }(p,1)_t^n\right\vert ^{p}}{\overline{\mu }(p,1)_t^n}\right) \left(
\frac{\left\vert \widehat{\sigma }(p,2)_t^n\right\vert ^{p}}{\overline{%
\mu }(p,2)_t^n}\right) Cov\left( \bar{V}\left( B^{H};p,1\right)^n_{t},\bar{V}\left( B^{H};p,2\right)^n_{t}\right) .
\end{eqnarray*}

(iv) This result follows immediately given parts (ii), (iii), and (iv) of Lemma \ref%
{LemmaBoot-New-moments}, the assumed condition $\left\vert\widehat{\sigma }(p,\upsilon)_t^n\right\vert ^{2p}\overset{u.c.p.}{\rightarrow }%
\int_{0}^{1}\left\vert \sigma _{s}\right\vert ^{2p}ds,$ and the definition
of $\widetilde{\Lambda }_{p,t}^n.$
\end{proof}

\begin{proof}[Proof of Theorem \ref{TheorJointCLT-a}]
\mikko{Note} that we
can write%
\begin{equation*}
\widehat{\mathbf{S}}_{n}^{\ast }=\widehat{A}_{n}^{\ast }\mathbf{S}%
_{n}^{\ast },
\end{equation*}%
where $\mathbf{S}_{n}^{\ast }$ is given by%
\begin{equation*}
\mathbf{S}_{n}^{\ast }=\left( \Sigma^{\ast }\left( X,B^{H};p\right)^n_{t}\right) ^{-1/2}\left( \Delta _{n}\right) ^{-1/2}\left(
\begin{array}{c}
\bar{V}^{\ast }\left( X,B^{H};p,1\right)^n_{t}-\mathbb{E}^{\ast
}\left( \bar{V}^{\ast }\left( X,B^{H};p,1\right)^n_{t}\right)
\\
\bar{V}^{\ast }\left( X,B^{H};p,2\right)^n_{t}-\mathbb{E}^{\ast
}\left( \bar{V}^{\ast }\left( X,B^{H};p,2\right)^n_{t}\right)%
\end{array}%
\right) ,
\end{equation*}%
and
\begin{equation*}
\widehat{A}_{n}^{\ast }=\left( \widehat{\Sigma }^{\ast }\left(
X,B^{H};p\right)^n_{t}\right) ^{-1/2}\left( \Sigma^{\ast
}\left( X,B^{H};p\right)^n_{t}\right) ^{1/2}.
\end{equation*}%
Hence, to obtain the desired result of $\widehat{\mathbf{S}}_{n}^{\ast
}, $ we may proceed in two steps:

\begin{description}
\item[Step 1.] Show that $\mathbf{S}_{n}^{\ast }\overset{d^{\ast }}{%
\rightarrow }N\left( 0,I_{2}\right) .$

\item[Step 2.] Show that $\widehat{A}_{n}^{\ast }\overset{\mathbb{P}^{\ast
}}{\rightarrow }I_{2}.$
\end{description}

For Step 1, note that we can write $\mathbf{S}_{n}^{\ast }$ as follows%
\begin{equation*}
\mathbf{S}_{n}^{\ast }=\left( \Sigma^{\ast }\left( X,B^{H};p\right)^n_{t}\right) ^{-1/2}D\cdot \mathbf{T}_{n}
\end{equation*}%
where%
\begin{equation*}
D=\left(
\begin{array}{cc}
\frac{\left\vert\widehat{\sigma }(p,1)_t^n\right\vert ^{p}}{\overline{%
\mu }(p,1)_t^n} & 0 \\
0 & \frac{\left\vert\widehat{\sigma }(p,2)_t^n\right\vert ^{p}}{%
\overline{\mu }(p,2)_t^n}%
\end{array}%
\right) ,
\end{equation*}%
and%
\begin{equation*}
\mathbf{T}_{n}=\left( \Delta _{n}\right) ^{-1/2}\left(
\begin{array}{c}
\bar{V}\left( B^{H};p,1\right)^n_{t}-\mathbb{E}\left( \bar{V}%
\left( B^{H};p,1\right)^n_{t}\right) \\
\bar{V}\left( B^{H};p,2\right)^n_{t}-\mathbb{E}\left( \bar{V}%
\left( B^{H};p,2\right)^n_{t}\right)%
\end{array}%
\right) .
\end{equation*}%
Under our assumptions, we have that \citep[Theorem 1]{BM83}
\begin{equation*}
\mathbf{T}_{n}\overset{d}{\rightarrow }N\left( 0,\Lambda _{p}\right) .
\end{equation*}%
Thus, results in Step 1 will follow if we can show that%
\begin{equation*}
\left( \Sigma^{\ast }\left( X,B^{H};p,\right)^n_{t}\right)
^{-1/2}D=\left( D^{-1}\left( \Sigma^{\ast }\left( X,B^{H};p\right)^n_{t}\right) ^{1/2}\right) ^{-1}\overset{\mathbb{P}^{\ast }}{%
\rightarrow }\Lambda _{p}^{-1/2}.
\end{equation*}%
To this end, note that we have
\begin{equation*}
D^{-1}\left( \Sigma^{\ast }\left( X,B^{H};p\right)^n_{t}\right) ^{1/2}=\left(
\begin{array}{cc}
\sqrt{\lambda _{p,n}^{11}} & 0 \\
\frac{\lambda _{p,n}^{12}}{\sqrt{\lambda _{p,n}^{11}}} & \sqrt{\lambda
_{p,n}^{22}-\frac{\left( \lambda _{p,n}^{12}\right) ^{2}}{\lambda _{p,n}^{11}%
}}%
\end{array}%
\right) \equiv \Theta _{p,n},
\end{equation*}%
where we use%
\begin{equation*}
D^{-1}\mathbf{=}\left(
\begin{array}{cc}
\frac{\overline{\mu }(p,1)_t^n}{\left\vert \widehat{\sigma }(p,1)_t^n\right\vert ^{p}} & 0 \\
0 & \frac{\overline{\mu }(p,2)_t^n}{\left\vert \widehat{\sigma }(p,2)_t^n\right\vert ^{p}}%
\end{array}%
\right) ,
\end{equation*}%
and%
\begin{eqnarray*}
&&\left( \Sigma^{\ast }\left( X,B^{H};p\right)^n_{t}\right) ^{1/2} \\
&=&\left(
\begin{array}{cc}
\sqrt{\left( \overline{\mu }(p,1)_t^n\right) ^{-2}\lambda
_{p,n}^{11}\left\vert\widehat{\sigma }(p,1)_t^n\right\vert ^{2p}} & 0 \\
\frac{\left( \overline{\mu }(p,2)_t^n\right) ^{-1}\lambda
_{p,n}^{12}\left\vert \widehat{\sigma }(p,2)_t^n\right\vert ^{p}}{\sqrt{%
\lambda _{p,n}^{11}}} & \sqrt{\left( \overline{\mu }(p,2)_t^n\right)
^{-2}\lambda _{p,n}^{22}\left\vert\widehat{\sigma }(p,2)_t^n\right\vert
^{2p}-\frac{\left( \overline{\mu }(p,2)_t^n\right) ^{-2}\left( \lambda
_{p,n}^{12}\right) ^{2}\left\vert \widehat{\sigma }(p,2)_t^n\right\vert
^{2p}}{\lambda _{p,n}^{11}}}%
\end{array}%
\right) .
\end{eqnarray*}%
The result follows since
\begin{equation*}
\Theta _{p,n}\Theta _{p,n}^{T}=\Lambda _{p,n}=\left(
\begin{array}{cc}
\lambda _{p,n}^{11} & \lambda _{p,n}^{12} \\
\lambda _{p,n}^{12} & \lambda _{p,n}^{22}%
\end{array}%
\right) \rightarrow \Lambda _{p}.
\end{equation*}%
For Step 2, it suffices to show that%
\begin{eqnarray*}
\left( \widehat{\Sigma }^{\ast }\left( X,B^{H};p\right)^n_{t}\right) ^{-1}\left( \Sigma^{\ast }\left( X,B^{H};p\right)^n_{t}\right)=\left( \left( \Sigma^{\ast }\left( X,B^{H};p\right)^n_{t}\right) ^{-1}\left( \widehat{\Sigma }^{\ast }\left(
X,B^{H};p\right)^n_{t}\right) \right) ^{-1}\overset{\mathbb{P}%
^{\ast }}{\rightarrow }I_{2}.
\end{eqnarray*}%
We here utilize the
fact that convergence in $L_{1}$ implies convergence in probability and that
all elements of the sum in $\widehat{\Sigma }_{i,j}^{\ast }$ are non-negative (where $%
\widehat{\Sigma }_{i,j}^{\ast }$ is the $\left( i,j\right) $-th element of
the matrix $\widehat{\Sigma }^{\ast }=\left( \widehat{\Sigma }%
_{i,j}^{\ast }\right) _{1\leq i,j\leq 2}$). In particular, given (\ref%
{Sigma-hat-1-star}) and (\ref{Boot-New-Assymp-Var-hat}), for $1\leq i,j\leq
2$, we have%
\begin{equation*}
\mathbb{E}^{\ast }\left\vert \widehat{\Sigma }_{i,j}^{\ast }\right\vert =%
\mathbb{E}^{\ast }\left( \widehat{\Sigma }_{i,j}^{\ast }\right) =\Sigma
_{i,j}^{\ast }.
\end{equation*}%
Thus, we
deduce that
\begin{equation*}
\widehat{\Sigma }^{\ast }\left( X,B^{H};p\right)^n_{t}-\Sigma^{\ast }\left( X,B^{H};p\right)^n_{t}\overset{%
\mathbb{P}^{\ast }}{\rightarrow }0.
\end{equation*}%
This concludes the proof of Step 2 and \mikko{also that} of Theorem \ref{TheorJointCLT-a}.
\end{proof}

\begin{proof}[Proof of Theorem \ref{Theor-Alpha-hat-boot-a}]
Given Theorem \ref{TheorJointCLT-a}, \mikko{the} result follows from an application of the delta method.
\end{proof}

\end{document}